\documentclass[11pt]{amsart}

\usepackage[colorlinks=true,citecolor=black!60!green,linkcolor=black!60!red,filecolor=black!60!cyan,urlcolor=black!60!magenta]{hyperref}
\usepackage[text={6.1in,8.5in},centering]{geometry}
\usepackage{amssymb,amsmath,amsthm,mathtools,extpfeil,mathscinet,textcomp}
\usepackage[all,cmtip]{xy}
\usepackage{tikz}
\usepackage[shortlabels]{enumitem}
\usepackage[shortcuts]{extdash}

\makeatletter
\newtheorem*{rep@theorem}{\rep@title}
\newcommand{\newreptheorem}[2]{%
\newenvironment{rep#1}[1]{%
 \def\rep@title{#2 \ref{##1}}%
 \begin{rep@theorem}}%
 {\end{rep@theorem}}}
\makeatother

\newtheorem{thm}{Theorem}
\newreptheorem{thm}{Theorem}
\newtheorem*{thm*}{Theorem}
\newtheorem*{lem*}{Lemma}

\newtheorem{lem}[thm]{Lemma}
\newtheorem{prop}[thm]{Proposition}
\newtheorem*{prop*}{Proposition}
\newtheorem{cor}[thm]{Corollary}
\newreptheorem{cor}{Corollary}

\theoremstyle{definition}

\newtheorem{ex}[thm]{Example}

\numberwithin{thm}{section}

\DeclareMathOperator{\id}{id}

\DeclareMathOperator{\vol}{vol}
\DeclareMathOperator{\mass}{mass}
\DeclareMathOperator{\supp}{supp}
\DeclareMathOperator{\rk}{rk}

\newcommand{\ph}{\varphi}
\newcommand{\epsi}{\varepsilon}

\newcommand{\CC}{\mathbb{C}}

\newcommand{\QQ}{\mathbb{Q}}
\newcommand{\NN}{\mathbb{N}}
\newcommand{\ZZ}{\mathbb{Z}}
\newcommand{\RR}{\mathbb{R}}
\newcommand{\HH}{\mathbb{H}}
\newcommand{\paren}[1]{\left({#1}\right)}
\newcommand{\abs}[1]{\left|{#1}\right|}

\newextarrow{\xrelbar}{5599}{\relbar\relbar{\relbar\vphantom{\rightarrow}}}

\begin{document}
\title{On elliptic and quasiregularly elliptic manifolds}
\author{Fedor Manin}
\author{Eden Prywes}
\begin{abstract}
    In his book \emph{Metric structures for Riemannian and non-Riemannian spaces}, Gromov defined two properties of Riemannian manifolds, \emph{ellipticity} and \emph{quasiregular ellipticity}, and suggested that there may be a connection between the two.  Since then, groups of researchers working independently have proved strikingly similar results about these two concepts.  We obtain new topological obstructions to the two properties: most notably, we show that closed manifolds of both types must have virtually abelian fundamental group.  We also give the first examples of open manifolds which are elliptic but not quasireguarly elliptic and vice versa.  Whether there is a direct connection between these properties---and, in particular, whether they are equivalent for closed manifolds---remains elusive.
\end{abstract}
\maketitle

\section{Introduction}
Language barriers between mathematical subfields often lead to the same ideas being developed independently in different contexts.  In this paper, we draw renewed attention to the parallels between elliptic and quasiregularly elliptic manifolds.  Both properties are defined by the existence of certain maps from Euclidean spaces; the two types of maps naturally lend themselves to different tools.  Although a link between the two notions was already suggested by Gromov in \cite[\S2.E and 6.D]{GrMS}, similar topological obstructions for ellipticity and quasiregular ellipticity have since been developed in parallel, among others by the two authors.  The similarity between both the statements and the proofs in the two cases supports Gromov's intuition.

We summarize the state of the field and prove several new results.  For example, we show that elliptic and quasiregularly elliptic closed manifolds must have abelian fundamental group; this follows from known restrictions on the topology of such manifolds.  On the other hand, we give a number of examples showing that, for open manifolds, ellipticity and quasiregular ellipticity do not coincide.

The most obvious open question remains: do they coincide for closed manifolds?

\subsection{Definitions and their motivations}
Let $M$ be an oriented Riemannian $n$-manifold.

A locally Lipschitz map $f:\RR^n \to M$ is said to have \emph{positive asymptotic degree} if
\[\limsup_{R \to \infty} R^{-n}\int_{B_R(\RR^n)} f^*d\vol_M>0.\]
Gromov \cite[2.41]{GrMS} suggested calling a manifold $M$ \emph{elliptic}\footnote{This terminology is inspired by, but should not be confused with, the rather distinct notion of elliptic spaces from rational homotopy theory.  We sometimes say ``Gromov-elliptic'' to emphasize this difference.} if there is a Lipschitz map $f:\RR^n \to M$ of positive asymptotic degree.  One might as well demand that $f$ be $1$-Lipschitz, in which case the map can be thought of as efficiently wrapping $M$ with an infinite sheet of wrapping paper; to be more concise, we will call such a map a \emph{wrapping map}.

Recently \cite{BeM,BeGM}, ellipticity has been shown to be crucial to understanding Lipschitz maps to $M$ from other manifolds.  For example, given a nullhomotopic $L$-Lipschitz map $N \to M$, is there a $C(L+1)$-Lipschitz nullhomotopy, for some constant $C=C(M,N)$?  This turns out to be true if $M$ is elliptic, simply connected, and \emph{formal}, a condition from rational homotopy theory, and false if $M$ is simply connected and formal but not elliptic.  Spaces satisfying all three conditions are called \emph{scalable}.  More general conjectures of Gromov about how the rational homotopy class of a map is constrained by its Lipschitz constant \cite[Ch.~7]{GrMS} are likewise confirmed for scalable spaces and false for spaces which are simply connected and formal but not scalable.

A map $f:N \to M$ between oriented Riemannian $n$-manifolds is said to be \emph{quasiregular} if $f \in W^{1,n}_{\text{loc}}(\RR^n,M)$ and there is a $K>0$ such that for almost every $x \in M$,
\[|Df(x)|^n \leq K\det Df(x).\]
Nonconstant quasiregular maps are discrete (in the sense that the preimage of every point is a discrete set) and open due to a theorem of Reshetnyak \cite{Reshet}.  They should therefore be thought of as generalized branched coverings with local geometric control.  A manifold is said to be \emph{quasiregularly} (or \emph{QR}) \emph{elliptic}, a term introduced by Bonk and Heinonen \cite{BH} following Gromov's template, if it admits a nonconstant quasiregular map from $\RR^n$.

The motivation for studying quasiregular maps comes from complex analysis, as the quasiregular condition is a relaxation of holomorphicity.  Holomorphic maps in one complex dimension satisfy
\[|Df(x)|^2=\det Df(x).\]
In other words, they are conformal except where they have zero derivative.  In higher dimensions, Liouville's theorem shows that maps satisfying this condition are either constant or restrictions of M\"obius transformations.  Studying quasiregular maps means understanding the degree to which relaxing this condition increases the universe of maps satisfying it.

Finally, a map $f:N \to M$ between metric spaces is said to have \emph{bounded length distortion} (\emph{BLD}) if it is locally bilipschitz along paths, that is, there is a constant $L$ such that for all paths $\gamma:[0,1] \to N$,
\[L^{-1}\operatorname{length}(\gamma) \leq \operatorname{length}(f(\gamma)) \leq L\operatorname{length}(\gamma).\]
Following Le Donne and Pankka \cite{LDPa}, we say that a Riemannian $n$-manifold $M$ is \emph{BLD-elliptic} if it admits a BLD map from $\RR^n$.  Since BLD maps $\RR^n \to M$ are both quasiregular and wrapping maps, every BLD-elliptic manifold is both elliptic and quasiregularly elliptic.  On the other hand, Pankka has suggested that every quasiregularly elliptic manifold may be BLD-elliptic.  This would explain many of the seeming coincidences observed in this paper.

It follows directly from each definition that if $\tilde M \to M$ is a Riemannian covering map, then $\tilde M$ is (Gromov-, quasiregularly, BLD-) elliptic if and only if $M$ is.  We use this fact without comment below.

\subsection{Obstructions to ellipticity and quasiregular ellipticity}
We know a number of results that restrict the topology of a closed elliptic or quasiregularly elliptic manifold.  Strikingly, all the results we discuss hold for both types, which suggests a possible close relationship between the two definitions.

One of the earliest known obstructions is the following:
\begin{thm}[{Varopoulos \cite[Theorem X.5.1]{vsc}, Gromov \cite[Theorem 2.42]{GrMS}}] \label{thm:pi1}
  If $M$ is an elliptic or quasiregularly elliptic closed $n$-manifold, then the growth rate of $\pi_1(M)$ is $O(R^n)$.  In particular, $\pi_1(M)$ is virtually nilpotent.
\end{thm}
The argument uses an isoperimetric inequality.  The elliptic case is easy to see.  Suppose that $f:\RR^n \to M$ is a Lipschitz map; it lifts to a map $\tilde f:\RR^n \to \tilde M$, where $\tilde M$ is the universal cover.  Then for every $R$, the volume of $\tilde f(\partial B_R(\RR^n))$ with multiplicity is $O(R^{n-1})$.  If sets $S \subset \tilde M$ satisfy
\[\vol_n(S)=o(\vol_{n-1}(\partial S)^{n/(n-1)}),\]
then the volume with multiplicity of $\tilde f(B_R(\RR^n))$ must be $o(R^n)$ and $f$ cannot be a wrapping map, so $M$ cannot be elliptic.  The same statement about isoperimetric inequalities can be shown for quasiregularly elliptic manifolds.  The statement about growth follows from a calculation due to Varopoulos.

The other main obstruction to ellipticity and quasiregular ellipticity is the following:
\begin{thm}[{Berdnikov--Guth--Manin \cite[Theorem 2.3]{BeGM}, Heikkil\"a--Pankka \cite{HP}}] \label{thm:coh}
  Let $M$ be an elliptic or quasiregularly elliptic closed $n$-manifold.  Then there is an injective ring homomorphism $H^*(M;\RR) \to \Lambda^* \RR^n$.
\end{thm}
In the elliptic case, this is a theorem of Berdnikov, Guth and Manin, who also show a quantitative refinement \cite[Theorem 2.3]{BeGM}.  The quasiregular result is a refinement of Prywes' main theorem in \cite{Pry} first shown by Heikkil\"a and Pankka \cite{HP}.  We give relatively concise and similar proofs of both statements by combining ideas from \cite{BeM} and \cite{Pry}.  In both cases, the main idea is to take pullbacks of arbitrary differential forms generating the cohomology ring and take a scaling limit.  While the forms may not obey the same relations as the cohomology classes, their scaling limits do, giving a ring homomorphism from $H^*(M;\RR)$ to an algebra of forms on $\RR^n$.  Evaluating this ring homomorphism at a point gives the desired result.

Most known obstructions to ellipticity and quasiregular ellipticity, including some obstructions to ellipticity derived by Gromov in \cite[\S2.E]{GrMS}, follow from Theorem \ref{thm:pi1} and Theorem \ref{thm:coh}.  The most far-reaching such consequence is the following, which answers a question of Gromov \cite[Question 2.44]{GrMS}:
\begin{thm} \label{cor:pi1}
  If $M$ is an elliptic or quasiregularly elliptic closed $n$-manifold, then $\pi_1(M)$ is virtually abelian.
\end{thm}
We show this by studying the cohomology of nilmanifolds.  By Theorem \ref{thm:pi1}, $\pi_1(M)$ is virtually nilpotent, and by Theorem \ref{thm:coh}, $H^*(M)$ must contain $H^*(\pi_1(M)^{\text{ab}})$ as a subalgebra.  On the other hand, any noncommutativity kills certain cup products between elements of $H^1(M)$.  We show this using Nomizu's theorem, which relates the cohomology of nilpotent groups to Lie algebra cohomology, but one could instead use the Serre spectral sequence.

The analogous result for BLD-elliptic manifolds was obtained using different methods by Le Donne and Pankka \cite{LDPa}.

We give several further corollaries in \S\ref{S:cors}, of which the easiest to state are the following.  Let $M$ be an elliptic or quasiregularly elliptic $n$-manifold.  Then:
\begin{itemize}
\item The fundamental group $\pi_1(M)$ cannot have rank $n-1$.
\item If $\pi_1(M)$ is infinite, then the Euler characteristic $\chi(M)=0$.
\end{itemize}
The latter is also shown in \cite{PG}; the authors use a different method (studying the growth of Betti numbers of finite covers), but their result is likewise a purely topological corollary of Theorem \ref{thm:pi1} and known results about the cohomology ring.

In fact, as far as we know, the only obstruction to ellipticity or quasiregular ellipticity of closed manifolds which \emph{does not} follow from Theorems \ref{thm:pi1} and \ref{thm:coh} is the following:
\begin{thm}[{Gromov \cite[Theorem 2.45]{GrMS}}] \label{thm:TsharpS}
  Suppose that $M$ is an elliptic closed $n$-manifold whose fundamental group is virtually $\ZZ^n$.  Then $M$ is aspherical (and therefore finitely covered by a manifold homeomorphic to the torus).
\end{thm}
The corresponding statement for quasiregularly elliptic manifolds is believed to be true, and an argument was given by Luisto and Pankka \cite{LuPa}.  However, after the first version of this article was posted, Pankka noticed that their argument only showed that $M$ must have a \emph{rationally} acyclic universal cover.

Theorem \ref{thm:TsharpS} is particularly interesting because unlike Theorem \ref{thm:coh}, it gives torsion information.  For example, it shows that $T^n \mathbin{\#} \Sigma$ is not elliptic, where $\Sigma$ is a nontrivial simply connected rational homology sphere.\footnote{Such a $\Sigma$ can be constructed as follows.  By attaching a $3$-cell to $S^2$ along a map of degree $2$, one creates a CW complex $X$ whose only nonzero homology group is $H_2(X) \cong \ZZ/2\ZZ$.  By embedding $X$ in $\RR^7$ and taking the boundary of a neighborhood of $X$, we obtain a $6$-manifold whose homology is the same as that of $X$ in low dimensions.  By Poincar\'e duality, this manifold is a rational homology sphere.}  The projection map from $T^n \mathbin{\#} \Sigma$ to $T^n$ is an isomorphism on rational cohomology.\footnote{The referee pointed out Peltonen's dissertation \cite{Pelt}, which produces obstructions to quasiregular ellipticity of manifolds $T^n \mathbin{\#} M$, where $M$ is again assumed to have a nontrivial rational cohomology class in some intermediate dimension.}

In contrast, for a large class of spaces, we prove the following result:
\begin{thm} \label{thm:Q-inv}
  If $M$ is a simply connected (or more generally nilpotent) closed manifold, then whether it is elliptic depends only on its rational homotopy type.
\end{thm}
For example, if there is a map $M \to N$ which induces isomorphisms on all rational homology groups, and $M$ and $N$ are both nilpotent or finitely covered by a nilpotent space, then $M$ is elliptic if and only if $N$ is.

A \emph{nilpotent space} is a space whose fundamental group is nilpotent and acts nilpotently on its higher homotopy groups.  A prototypical example would be the mapping torus of a self-map of $(S^n)^r$ which induces a unipotent transformation on $H_n$.

In particular, the higher homotopy groups of a nilpotent space are all finitely generated.  On the other hand, for the manifold $T^n\mathbin{\#}\Sigma$ discussed above, if $k>1$ is minimal such that $H_k(\Sigma) \neq 0$, then (writing $\widetilde{M}$ for the universal cover of $M$)
\[\pi_k(T^n\mathbin{\#}\Sigma) \cong \pi_k(\widetilde{T^n\mathbin{\#}\Sigma}) \cong H_k(\widetilde{T^n\mathbin{\#}\Sigma}) \cong \ZZ^n[H_k(\Sigma)],\]
so it is not nilpotent.  In fact, though, the key ingredient of Theorem \ref{thm:TsharpS} seems to be that the fundamental group has full rank, since other similar non-nilpotent examples are elliptic:
\begin{thm} \label{thm:TxS}
    Let $X=(T^{n-k} \times S^k) \mathbin{\#} \Sigma$, where $2 \leq k \leq n-1$ and $\Sigma$ is a simply connected rational homology sphere.  Then $X$ is elliptic.
\end{thm}

\subsection{Construction of elliptic and quasiregularly elliptic manifolds}
In general, wrapping maps are more flexible than quasiregular maps, and therefore it is easier to show that a manifold is Gromov-elliptic than quasiregularly elliptic.  For example, this follows directly from the definition:
\begin{prop}
    Let $M$ and $M'$ be elliptic manifolds.
    \begin{itemize}
    \item The product $M \times M'$ is elliptic.
    \item If $M$ is closed, $N$ is a closed manifold of the same dimension, and there is a map of positive degree $M \to N$, then $N$ is also elliptic.
    \end{itemize}
\end{prop}
A general construction is given in \cite{BeGM} for formal manifolds.  A simply connected space is \emph{formal} in the sense of Sullivan \cite[\S12]{SulLong} if, informally speaking, its rational homotopy type is the simplest one with its rational cohomology ring.  An equivalent informal condition is that there are no nontrivial higher-order rational cohomology operations: the only relations between cohomology classes are given by the cup product.  For finite complexes, another useful equivalent condition is as follows: a simply connected finite CW complex $X$ is formal if and only if there is an integer $p>1$ and a map $f_p:X \to X$ whose induced map on $H^k(X;\QQ)$, for every $k$, is multiplication by $p^k$.

Examples of formal manifolds include all simply connected symmetric spaces, such as spheres, complex projective spaces, and Grassmannians.  Products and connected sums of formal manifolds are also formal.  The simplest and lowest-dimensional example of a non-formal simply connected manifold is the total space $M$ of an $S^3$-bundle over $S^2 \times S^2$ \cite[p.~94]{FOT}, further discussed in \S\ref{S:exs}.
\begin{thm}[{\cite[Theorem C$'$]{BeGM}}]
    Let $M$ be a formal simply connected closed $n$-manifold.  Then $M$ is elliptic if and only if there is an embedding $H^*(M;\RR) \hookrightarrow \Lambda^*\RR^n$.
\end{thm}
In particular, since every simply connected manifold of dimension $\leq 6$ is formal, this gives a complete classification for such manifolds.

We note that if there is an embedding $H^*(M;\QQ) \hookrightarrow \Lambda^*\QQ^n$, then there is a map of positive degree from $T^n$ to $M$, and therefore $M$ is trivially elliptic.  However, a real embedding as above need not be rationalizable; a 9-dimensional example manifold $M$ is given in \cite[Remark 5.2]{BeGM} and we review this example in \S\ref{S:exs}.  A wrapping map of this $M$ must therefore be in some sense aperiodic (but may be almost periodic).

In contrast, we know few examples of quasiregularly elliptic manifolds outside low dimensions.  What's more, all the examples we know factor through a torus; even understanding whether the specific $9$-manifold alluded to above is quasiregularly elliptic would be a breakthrough.  In low dimensions, a complete classification is available and coincides with that of elliptic manifolds:
\begin{itemize}
\item In dimension $2$, only $S^2$ and $T^2$ are quasiregularly elliptic.  (In fact, every quasiregular map from $\CC$ to a Riemann surface or a domain in $\CC$ decomposes into a homeomorphism followed by a holomorphic map, a fact known as \emph{Sto\"ilow factorization}.)
\item In dimension $3$, only $S^3$, $S^2 \times S^1$, $T^3$, and their finite quotients are quasiregularly elliptic due to work of Jormakka \cite{Jor} and the geometrization theorem.
\item In dimension $4$, Heikkil\"a and Pankka \cite{HP} give a classification up to homeomorphism in the simply connected case: all those manifolds which are not excluded by Theorem \ref{thm:coh} are homeomorphic to quasiregularly elliptic ones by work of Piergallini and Zuddas \cite{PZ} because they admit a branched covering by $T^4$.\footnote{Earlier, Rickman had constructed such a branched covering for $(S^2 \times S^2) \mathbin{\#} (S^2 \times S^2)$ \cite{RickSum}.}  Specifically, these are
  \[\#^r (S^2 \times S^2), \qquad \#^s \CC P^2 \#^t \overline{\CC P^2}, \qquad 0 \leq r,s,t \leq 3,\]
  where, by convention, the connected sum of nothing is $S^4$.  We extend this topological classification to all $4$-manifolds in Corollary \ref{cor:4D}: up to finite quotients, the only additional examples are $T^4$, $T^2 \times S^2$, and $S^1 \times S^3$.
    
  Note that this is not a complete classification of which $4$-manifolds are quasiregularly elliptic, since quasiregular ellipticity is only known to be invariant up to PL homeomorphism.  Some of the manifolds above are known to admit exotic smooth (and therefore PL) structures \cite{AP1,AP2}; it is not known whether any such exotic $4$-manifolds are quasiregularly elliptic.\footnote{Recent progress in resolving this has been made in a preprint by Bais, Piergallini and Zuddas \cite{BPZ}.  They show that such manifolds admit a branched covering by $T^4$ if they have a handle decomposition without $1$- or $3$-handles.  Whether such a decomposition exists for all closed simply connected $4$-manifolds is an open question \cite[Problem 4.18]{Kirby}, though counterexamples for compact manifolds with boundary were given by Hayden and Piccirillo \cite[Cor.~A.2]{HayPic}.}
\end{itemize}
In higher dimensions, it is known that $\CC P^n$ is quasiregularly elliptic, as shown by Luisto and Prywes \cite[Theorem 1.3]{LuPr}.

\subsection{Open manifolds}
Thus far we have discussed Gromov- and quasiregular ellipticity for closed manifolds.  This restriction is natural because it makes the question topological: for closed manifolds, ellipticity is a homotopy invariant,\footnote{This follows from the fact that any homotopy equivalence between compact Riemannian manifolds can be made Lipschitz.} and quasiregular ellipticity is invariant at least under PL homeomorphism \cite[Proposition B.1]{HP}.  On the other hand, for open manifolds both properties depend on the metric, and in different ways, as we discuss further in \S\ref{S:open-d2} and \ref{S:open}.

In prior works on quasiregular ellipticity such as \cite{PaRa,PaRaWu}, an open manifold is said to be quasiregularly elliptic if it admits a quasiregular map from $\RR^n$ with respect to \emph{some} Riemannian metric.  This forces quasiregular ellipticity to be a topological property.  Even with this definition, Theorem \ref{thm:pi1} extends to the open manifold case:
\begin{thm}[{Pankka--Rajala \cite[Theorem 1.3]{PaRa}}] \label{thm:pi1-open-QR}
  If $M$ is a quasiregularly elliptic (closed or open) Riemannian $n$-manifold, then the growth rate of any finitely generated subgroup of $\pi_1(M)$ is $O(R^n)$.
\end{thm}
\noindent Note that $\pi_1(M)$ itself may be infinitely generated (e.g.~$\QQ$); in that case, this result still forces it to be nilpotent of step at most $\sim \sqrt n$.

A drawback of this topological definition is that, for example, it erases the difference between the Euclidean and hyperbolic planes, which are in different conformal classes.  In the two-dimensional case, the uniformization theorem allows one to fully classify which Riemannian manifolds are quasiregular, up to isometry.

Moreover, in this sense every open manifold would be elliptic, as it is easy to modify any Riemannian metric on an open manifold to make it elliptic.  This gives examples of elliptic open manifolds which cannot be quasiregularly elliptic.  The easiest such examples have infinite volume, but we also give a construction of finite-volume examples.  To prove a version of Theorem \ref{thm:pi1-open-QR} for ellipticity, we must impose an additional restriction:
\begin{thm} \label{thm:pi1-open}
    Suppose that $M$ is a Riemannian $n$-manifold of finite volume, and suppose that its volume form can be written as $\omega+d\alpha$ where $\omega$ is compactly supported and $\alpha$ is bounded.  If $M$ is elliptic, then every finitely generated subgroup of $\pi_1(M)$ has growth $O(R^n)$.
\end{thm}
The hypothesis implies that any Lipschitz map from $\RR^n$ \emph{equidistributes}, that is, at large scales it ``visits'' any neighborhood roughly equally often---something which is always true of quasiregular maps.  It holds, for example, for open subsets of closed manifolds, as well as for manifolds of pinched negative curvature and nonpositively curved locally symmetric spaces.  The easiest non-examples are infinite strings of surfaces joined by thinner and thinner necks.  In future work it will be shown that it is equivalent to $M$ having a nonzero Cheeger constant.

Even then, there are manifolds of dimension $n \geq 3$ which have an elliptic metric satisfying this hypothesis, but no quasiregularly elliptic metric.  One such example is $S^n \setminus X$ where $X$ is a countable subset of the equatorial $S^{n-2}$.  We show in Theorem \ref{codim n-2} that $S^n \setminus S^{n-2}$ (with the round metric) is elliptic, and therefore so is $S^n \setminus X$.  On the other hand, a generalization of the Rickman--Picard theorem \cite{Rick} by Holopainen and Rickman \cite[Theorem 3.1]{HoR2} shows that $S^n \setminus X$ cannot be quasiregularly elliptic in the topological sense above.

Finally, we show in \S\ref{S:open-d2} that even for metrics on $\RR^2$, there is no implication between ellipticity and quasiregular ellipticity.  In particular, we give the only known example of a Riemannian manifold which is quasiregularly elliptic, but not elliptic.

\subsection{Are closed quasiregularly elliptic manifolds elliptic?}
This would be the case if every such manifold were BLD-elliptic, since every BLD map is both quasiregular and Lipschitz.  There is not much evidence for this conjecture; Gromov \cite[6.42]{GrMS} sketches a proof, further discussed by Bonk and Heinonen \cite[\S2]{BH}, that if there is a quasiregular map $\RR^n \to M$, then in fact there is a H\"older quasiregular map\footnote{The case $M=S^n$ was done by Miniowitz \cite{Mini}.}, but it's not clear how one would upgrade this to a Lipschitz quasiregular map.

In general, a quasiregular map need not be even locally Lipschitz: for example, consider the radial stretch map $f:\RR^n \to \RR^n$ defined by $f(x)=|x|^\alpha \cdot x$, for $\alpha<1$.  Moreover, there are examples of quasiregular maps with wild branch sets which are not even topologically conjugate to Lipschitz maps \cite{HR}.  On the other hand, if one assumes that the branch set (the subset of the domain at which the map is not a local homeomorphism) is contained in a subcomplex of codimension $2$, then the map is always topologically conjugate to a locally Lipschitz one, namely a PL branched cover, due to work of Luisto and Prywes \cite{LuPr}.  Thus, even if one wants to show that quasiregularly elliptic manifolds are BLD-elliptic---a more restrictive condition than Gromov ellipticity---the main obstacle is still likely to be the global, rather than local, geometry of the map.

\subsection{Are closed elliptic manifolds quasiregularly elliptic?}
A priori, this may seem unlikely because quasiregularity is such a restrictive condition (for example, in that quasiregular maps are always open and discrete).  However, while it is false for open manifolds as discussed above, the question remains open for closed manifolds.

Gromov suggested in \cite[p.~67]{GrMS} that it might be possible to obtain a quasiregular map by somehow optimizing over all Lipschitz maps of positive asymptotic degree: for example, looking at $1$-Lipschitz maps which attain the highest possible degree on every $R$-ball.  The example of $S^2$ is illustrative.  Here, for any choice of rectangle on the equator one gets a quasiregular map which is a cover branched at the preimages of the vertices of the rectangle and sends alternate squares of a checkerboard to the north and south hemisphere.  If one gives the squares side length $\pi$, then this map can be made $1$-Lipschitz, and presumably has the highest possible degree on all balls.  On the other hand, one can achieve the same degree on balls if one forces two branching points to collide with each other, at which point entire lines of the checkerboard are mapped to the point of collision, and the map is no longer quasiregular.  Moreover, at that point one can start moving whole strips of the checkerboard laterally, giving many more degrees of freedom of travel through maps that do not seem to  even be close to a quasiregular map.  In other words, we have described a large continuous family of optimal maps, of which a non-dense subset is quasiregular.

\subsection{Key examples for future work} \label{S:exs}
There are three key examples of spaces whose behavior must be understood to know the full picture.
\begin{enumerate}
\item Are the examples of non-nilpotent elliptic manifolds given in Theorem \ref{thm:TxS} also quasiregularly elliptic?
\item The simplest example of a simply connected non-formal manifold is the $7$-manifold $M$ which completes the pullback square
  \[\xymatrix{
    M \ar[r] \ar[d] & S^7 \ar[d]^{\text{Hopf fibration}} \\
    S^2 \times S^2 \ar[r]^-{\text{deg }1} & S^4.
  }\]
  Is this manifold elliptic?  We suspect not, but do not know how to prove this.

  Every elliptic formal $n$-manifold has $L$-Lipschitz self-maps of degree $\geq CL^n$, for all sufficiently large $L$.  On the other hand, the degree of an $L$-Lipschitz self-map of $M$ is at most $L^{20/3}$ for homological reasons.  But the same argument does not give an obstruction to building efficient Lipschitz maps $\RR^7 \to M$; indeed, we have a sketch that such a map can be built with
  \[\int_{B(R)} f^*d\vol_M \geq L^7/\log L.\]
  But the logarithmic factor seems difficult, and perhaps impossible, to get rid of.
\item For all known QR elliptic closed manifolds, the quasiregular map is constructed by factoring through a torus.  Therefore, the existence of elliptic manifolds which do not admit positive degree maps from a torus poses a challenge for constructing new quasiregular maps.

  One such space is constructed in \cite[Remark 5.2]{BeGM} as follows.  Take the real Poincar\'e duality space
  \[Y=(S^2 \times S^2)/(x,*) \sim (*,x) \mathbin{\#} 2\mathbb CP^2 \mathbin{\#} 3 \overline{\mathbb CP^2},\]
  where $*$ is a basepoint.  The cup product $H^2(Y) \times H^2(Y) \to H^4(Y)$ is the quadratic form $\langle 2,1,1,-1,-1,-1 \rangle$, which has discriminant $-2$, and therefore is not rationally equivalent to the quadratic form induced by the cup product $H^2(T^4) \times H^2(T^4) \to H^4(T^4)$.  Embed $Y$ in $\RR^{10}$.  The boundary of a thickening of this embedding is a $9$-manifold with the desired property.
\end{enumerate}

\section{Proof of the cohomological obstruction}
In this section we prove Theorem \ref{thm:coh}.  We start with the elliptic version, first given in \cite{BeGM}.
\begin{thm}
    Let $M$ be a Gromov-elliptic closed $n$-manifold.  Then there is an injective ring homomorphism $H^*(M;\RR) \to \Lambda^* \RR^n$.
\end{thm}
The basic idea, which carries over into the proof of the quasiregular case, is to choose differential forms generating the cohomology ring of $M$ and pull them back along a Lipschitz map of positive asymptotic degree.  As we ``zoom out'' further and further and rescale, the exact forms corresponding to relations between these forms become smaller and smaller.  In the limit, the differential forms factor through a ring homomorphism from $H^*(M;\RR)$ to a space of forms.  Restricting to a point where this homomorphism is nontrivial, we get a homomorphism to $\Lambda^*\RR^n$.

For the elliptic case, we use flat differential forms as introduced by Whitney \cite[Ch.~IX]{Whi}.  These can be thought of as formal limits of smooth forms in the \emph{flat norm} given by
\[\lVert\omega\rVert_\flat=\lVert\omega\rVert_\infty+\lVert d\omega \rVert_\infty.\]
Flat forms have a number of relevant properties, for example the pullback of a smooth (or flat) form along a Lipschitz function is a flat form \cite[Theorem 3.6]{GKS}.  Following Whitney, we write $\Omega_\flat^*(M)$ for the complex of flat forms on $M$.
\begin{proof}
  Let $w_1,\ldots,w_r$ be a generating set for $H^*(M;\RR)$, where the degree of $w_i$ is $a_i$, and for each $i$ let $\omega_i \in \Omega^{a_i}(M)$ be a differential form representing $w_i$.  We include the fundamental class among the $w_i$, and represent it by the volume form $d\vol_M$.  Let $P_1,\ldots,P_s \in \RR[w_1,\ldots,w_r]$ be a list of polynomials such that as a ring,
  \[H^*(M;\RR) \cong \RR[w_1,\ldots,w_r]/(P_1,\ldots,P_s).\]
  (Here $\RR[w_1,\ldots,w_r]$ denotes the free \emph{graded commutative} ring in variables of the appropriate degrees.)  Then for each $j$, $P_j(\omega_1,\ldots,\omega_r)$ is an exact form, so we can fix $\alpha_j \in \Omega^{b_j-1}(M)$ such that
  \[d\alpha_j=P_j(\omega_1,\ldots,\omega_r),\]
  where $b_j$ is the (graded) degree of $P_j$.

  Suppose that $f:\RR^n \to M$ is a $1$-Lipschitz map of positive asymptotic degree.  For every $t>0$ define $f_t(x)=f(tx)$; this is a $t$-Lipschitz map.  Now we consider forms
  \[\omega_{i,t}=\frac{f_t^*\omega_i}{t^{a_i}}, \qquad \alpha_{j,t}=\frac{f_t^*\alpha_j}{t^{b_j}}.\]
  Since pulling back along a $t$-Lipschitz map multiplies the infinity-norm of a $k$-form by at most $t^k$, we have
  \[\lVert \omega_{i,t} \rVert_\infty \leq 1, \qquad \lVert \alpha_{j,t} \rVert_\infty \leq 1/t.\]
  By definition of positive asymptotic degree, there is an $\epsi>0$ and a sequence of $t \to \infty$ such that $t^{-n}\int_{B_1(\RR^n)} f_t^*d\vol_M \geq \epsi$.  By the Arzel\`a--Ascoli theorem, this sequence has a subsequence $t_1,t_2,\ldots \to \infty$ for which the $\omega_{i,t_k}$ converge in the flat norm; we have
  \[\lim_{k \to \infty} \omega_{i,t_k}=\omega_{i,\infty} \in \Omega^{a_i}_\flat(\RR^n), \qquad \lim_{k \to \infty} \alpha_{j,t_k}=0.\]
  Since flat limits preserve wedge products \cite[\S IX.17]{Whi} and differentials (by definition), this implies that the ring homomorphism $\RR[w_1,\ldots,w_r] \to \Omega^*_\flat(\RR^n)$ defined by $w_i \mapsto \omega_{i,\infty}$ passes to a well-defined map on the quotient ring $H^*(M,\RR) \cong \RR[w_1,\ldots,w_r]/(P_1,\ldots,P_s)$, giving a ring homomorphism
  \[\ph_\infty:H^*(M;\RR) \to \Omega^*_\flat(\RR^n).\]
  Moreover, flat convergence implies that
  \[\int_{B_1(\RR^n)} \ph_\infty(d\vol_M) \geq \epsi.\]
  In particular, $\ph_\infty(d\vol_M)$ is nonzero on some set of positive measure.  While flat forms are not well-defined pointwise, they are well-defined up to a measure zero set, so we can choose representatives and then choose a point in this set of positive measure where these representatives actually restrict to a ring homomorphism
  \[H^*(M;\RR) \to \Lambda^*\RR^n.\]
  This homomorphism sends the fundamental class to a nonzero element, so by Poincar\'e duality it is injective.
\end{proof}
The proof of the quasiregular case uses a similar overall strategy, but is somewhat more complicated, using ideas from \cite{Pry}.  The main idea is still to compare forms defined on larger and larger balls, so that in the limit primitives go to zero.  However, the scaling factor is no longer related to the size of the ball; instead, the existence of a useful scaling factor is guaranteed by quasiregularity.  Moreover, the sequence converges in an $L^p$ rather than a flat sense.

Before we begin the proof, we set up the main tools we will use.  In order to apply results from \cite{Pry}, we assume that $M$ has a nontrivial cohomology class in some degree strictly between $0$ and $n$; if this is not the case, the theorem is trivial.  As before, let $w_1,\ldots,w_r$ be a generating set for $H^*(M;\RR)$, where the degree of $w_i$ is $a_i$, and for each $i$ let $\omega_i \in \Omega^{a_i}(M)$ be a differential form representing $w_i$.  Once again, we assume that one of the $w_i$ is the fundamental class and the corresponding $\omega_i$ is $d\vol_M$.  Let $P_1,\ldots,P_s \in \RR[w_1,\ldots,w_r]$ be a list of polynomials such that as a ring,
\[H^*(M;\RR) \cong \RR[w_1,\ldots,w_r]/(P_1,\ldots,P_s).\]
Then for each $j$, $P_j(\omega_1,\ldots,\omega_r)$ is an exact form, so we can fix $\alpha_j \in \Omega^{b_j-1}(M)$ such that
\[d\alpha_j=P_j(\omega_1,\ldots,\omega_r),\]
where $b_j$ is the (graded) degree of $P_j$.

Suppose that $f \colon \RR^n \to M$ is $K$-quasiregular.  For $B\subset \RR^n$ a ball, let 
\[A(B)=\int_B J_f = \int_{B} f^*d\vol_M\]
be the \textit{counting function} for $f$.  We consider rescalings of the forms $f^*\omega_i$ and $f^*\alpha_j$ by this counting function, as described in \cite[Section 3]{Pry}.  Given a sequence of balls $B_k \subset \RR^n$ and an $s$-form $\eta \in \Omega^s(M)$, we define forms on $B = B_1(\RR^n)$ by
\begin{align*}
  \rho_k(\eta) &= \frac{1}{A(B_k)^{s/n}} T_k^*f^*\eta \\
  \sigma_k(\eta) &= \frac{1}{A(B_k)^{(s+1)/n}}T_k^*f^*\eta,
\end{align*}
where $T_k$ is the homothety taking $B$ to $B_k$.

We summarize the main results in \cite[Section 3]{Pry} as applied to our setting.  Define $p_i = n/k$ and $q_j = n/b_j$.

\begin{lem}\label{lem:lpconvergence}
  There exists a sequence of balls $B_k \subset \RR^n$ for which the following
  holds.
  \begin{enumerate}[(i)]
  \item Whenever $0<a_i<n$, there is a form $\rho_\infty(\omega_i)$ so that
    \[\rho_k(\omega_i) \to \rho_\infty(\omega_i)\]
    weakly in $L^{p_i}(B)$.  Moreover, there exist forms $\tau_{i,k}$ and
    $\tau_{i,\infty}$ so that $d\tau_{i,k} = \rho_k(\omega_i)$,
    $d\tau_{i,\infty} = \rho_\infty(\omega_i)$, and $\tau_{i,k}$ converges strongly
    in $L^{p_i}(B)$ to $\tau_{i,\infty} \in W^{1,p_i}(B)$. \label{parti}
  \item For every ball $B_k$, we have that $A(B_k) \le C(n) A(\frac{1}{2}B_k)$,
    and $\lim_{k\to \infty} A(B_k) = \infty$. \label{partiii}
  \item The forms $\rho_k(d\vol_M)$ converge in distribution to a form
    $\rho_\infty(d\vol_M)$ which satisfies \label{partiv}
    \[\int_B \rho_\infty(d\vol_M)=\vol(B).\]
  \end{enumerate}
\end{lem}
The first new result relates to the convergence of rescalings of $\alpha_j$:
\begin{lem}\label{lem:alphaconvergence}
  Up to taking a subsequence of the index $k$, the sequence
  $\sigma_k(\alpha_j) \to 0$ strongly in $L^{\frac{n}{b_j-1}}(B)$ and
  $d\sigma_k(\alpha_j) \to 0$ weakly in $L^{q_j}(B)$ when $b_j<n$ and in
  distribution when $b_j = n$.
\end{lem}
\begin{proof}
  We first show that $\sigma_k(\alpha_j)$ converges to $0$:
  \begin{align*}
    \|\sigma_k(\alpha_{j})\|_{\frac{n}{b_j-1}} &\le A(B_k)^{-\frac{b_j}{n}}\paren{\int_{B_k} |f^*\alpha_{j}|^{\frac{n}{b_j-1}}}^{\frac{b_j-1}{n}} \\
    &\le A(B_k)^{-\frac{b_j}{n}} \|\alpha_j\|_\infty K \paren{\int_{B_k} J_f}^{\frac{b_j-1}{n}} \\
    &= \|\alpha_j\|_\infty K A(B_k)^{-\frac{1}{n}},
  \end{align*} 
  which converges to $0$ as $k \to \infty$ by Lemma \ref{lem:lpconvergence}\ref{partiii}.

  We next show that $d\sigma_k(\alpha_j)$ converges in distribution.  Suppose
  that $\phi$ is a test form on $B$.  Then using integration by parts,
  \begin{align*}
    \abs{\int_{B} d\sigma_k(\alpha_{j}) \wedge \phi} = \abs{-\int_{B}\sigma_k(\alpha_j) \wedge d\phi}.
  \end{align*}
  This term converges to $0$ as $k \to \infty$ by the strong convergence of
  $\sigma_k(\alpha_{j})$.

  Lastly,
  \begin{align*}
    \|d\sigma_k(\alpha_{j})\|_{\frac{n}{b_j}} &\le K \|d\alpha_{j}\|_\infty A(B_k)^{-b_j/n} \paren{\int_{B_k}J_f}^{-b_j/n} \\
        &=K \|d\alpha_{j}\|_\infty.
  \end{align*}
  Therefore the sequence is bounded, and when $b_j<n$ there exists a
  subsequence that converges weakly.  Since the sequence converges to $0$ in
  distribution, the weak limit is also $0$.
\end{proof}
Next we show that wedge products of rescaled forms converge to the wedge
product of their limits.
\begin{lem}\label{lem:productconvergence}
  Let $\omega_i$ be one of the closed forms above and let $\eta_{k}$ be a
  sequence of closed $a'$-forms, where $a_i+a' \leq n$ and $a_i,a' \neq 0$. Let
  $p' = n/a'$.  If $\eta_k \to \eta_{\infty}$ weakly in $L^{p'}(B)$, then
  $\rho_k(\omega_i) \wedge \eta_k \to \rho_\infty(\omega_i) \wedge \eta_\infty$
  weakly in $L^{n/(a_i+a')}(B)$, if $a_i+a' < n$, and in distribution, if $a+a' = n$.
  When $a_i+a'=n$, the element lies in $L^1(B)$.
\end{lem}
\begin{proof}
  Let $\phi$ be a smooth test form on $B$.  By Lemma \ref{lem:lpconvergence}\ref{parti},
  we have forms $\tau_{i,k}$ and $\tau_{i,\infty}$ satisfying
  $d\tau_{i,k} =\rho_k(\omega_i)$ and $d\tau_{i,\infty} = \rho_\infty(\omega_k)$, and such that $\tau_{i,k} \to \tau_{i,\infty}$ strongly in $L^p$.
  Then
  \begin{align*}
    \abs{\int_{B} (\rho_k(\omega_i) \wedge \eta_k - \rho_\infty(\omega_i) \wedge \eta_\infty) \wedge \phi}  &= \abs{\int_{B} d(\tau_{i,k} \wedge \eta_k - \tau_{i,\infty} \wedge \eta_\infty) \wedge \phi} \\
            &= \abs{\int_{B} (\tau_{i,k} \wedge \eta_k - \tau_{i,\infty} \wedge \eta_\infty) \wedge d\phi},
  \end{align*}
  by integration by parts. Additionally,
  \[\abs{\int_{B} (\tau_{i,k} \wedge \eta_k - \tau_{i,\infty} \wedge \eta_\infty) \wedge d\phi}  \le \abs{\int_{B} (\tau_{i,k} - \tau_{i,\infty}) \wedge (\eta_k \wedge d\phi)} + \abs{\int_{B}\tau_{i,\infty}\wedge(\eta_k - \eta_\infty)\wedge d\phi}.\]
  For the first term, by H\"older's inequality,
  \[\abs{\int_{B} (\tau_{i,k}-\tau_{i,\infty}) \wedge (\eta_k \wedge d\phi)} \le C(n)\|\tau_{i,k} - \tau_{i,\infty}\|_{p_i}\|\eta_k\|_{p'} \|d\phi\|_q,\]
  where $q^{-1} = 1- p_i^{-1} - (p')^{-1}$.  This converges to $0$ since $\tau_{i,k}$ converges strongly and $\eta_k$ is bounded in $L^{p'}$.  For
  the second term, $d\phi\wedge\tau_{i,\infty}$ is an element of $L^{p'/(p'-1)}$
  since $a_i+a' \le n$. By the weak convergence of $\eta_k$, this term also
  converges to $0$.  This gives that $\rho_k(\omega_i) \wedge \eta_k$ converges
  to $\rho_\infty(\omega_i) \wedge \eta_\infty$ in distribution.  If $a_i+a' < n$,
  by H\"older's inequality we know that 
  \[\|\rho_k(\omega_i) \wedge \eta_k\|_{\frac{n}{a_i+a'}} \le \|\rho_k(\omega_i)\|_{p_i} \|\eta_k\|_{p'}.\]
  Therefore the wedge product converges weakly, up to a subsequence, to some
  element in $L^{n/(a_i+a')}(B)$.  This element must be
  $\rho_\infty(\omega_i) \wedge \eta_\infty$ by the convergence result in
  distribution. When $a_i+a' = n$, it lies in $L^1(B)$.
\end{proof}

\begin{proof}[Proof of Theorem \ref{thm:coh}]
  Let $L^*(B)$ be the ring of differential forms whose elements of degree $a$
  are in $L^{n/a}(B)$.  In order to prove the theorem, we would like to show
  that the ring homomorphism
  \[\ph_\infty: \RR[w_1,\ldots,w_r] \to L^*(B)\]
  defined by $\ph_\infty(w_i)=\rho_\infty(\omega_i)$ factors through a
  homomorphism
  \[\overline{\ph_\infty} \colon H^*(M,\RR) \to L^*(B).\]
  Inductively applying Lemma \ref{lem:productconvergence}, we see that for
  every polynomial $P(w_1,\ldots,w_r)$,
  \[\ph_\infty(P(w_1,\ldots,w_r))=\lim_{k \to \infty} \rho_k(P(\omega_1,\ldots,\omega_r))\]
  in the weak topology of $L^{n/\deg P}(B)$, if $\deg P<n$, and in distribution,
  if $\deg P=n$.  By Lemma \ref{lem:alphaconvergence}, it follows that for each
  defining relation $P_j$ of $H^*(M;\RR)$,
  \[\ph_\infty(P_j(w_1,\ldots,w_r))=\lim_{k \to \infty} \rho_k(d\alpha_j)=
    \lim_{k \to \infty} d\sigma_k(\alpha_j)=0,\]
  so indeed the homomorphism factors through.

  Next we must show that $\ph_\infty$ is injective.  As in the elliptic case, it
  suffices to show that the fundamental class is sent to a nonzero element.
  But this is Lemma \ref{lem:lpconvergence}\ref{partiv}.
\end{proof}

\section{Consequences of the cohomological obstruction} \label{S:cors}

\subsection{The fundamental group}
We start by proving Corollary \ref{cor:pi1}, restated here:
\begin{repthm}{cor:pi1}
  If $M$ is a Gromov-elliptic or quasiregularly elliptic closed $n$-manifold, then $\pi_1(M)$ is virtually abelian.
\end{repthm}
This can be derived from a purely algebraic lemma.  Given a Lie group $G$, we define its cohomology ring $H^*(G)$ as the cohomology of the differential graded algebra (DGA) $\Omega^*(G)^G$ of $G$-invariant differential forms on $G$.  This DGA $\Omega^*(G)^G$ is also known as the \emph{Chevalley--Eilenberg complex} of the Lie algebra $\mathfrak g$ of $G$.  As a ring, $\Omega^*(G)^G \cong \Lambda^*\RR^{\dim G}$, since every element is obtained by translating a form on the tangent space of the identity.  However, the differential is determined by the group structure of $G$.  The invariant vector fields on $G$ can be identified with $\mathfrak g$, and the differential on invariant $1$-forms is dual to its Lie bracket, thought of as a function $\mathfrak g \wedge \mathfrak g \to \mathfrak g$.
\begin{lem} \label{lem:Lie}
  Let $G$ be a simply connected nilpotent Lie group.  If there is an injective
  ring homomorphism $i:H^*(G) \to \Lambda^*\RR^n$ for some $n$, then $G$ is
  abelian.
\end{lem}
This can be thought of as a special case of a Lie group analogue of a theorem of Stallings for discrete groups \cite{Sta}.
\begin{proof}
  Consider the lower central series
  \[G_0=G, \quad G_1=[G,G], \quad G_2=[G_1,G], \quad \ldots\]
  and the corresponding tower of quotients
  \[\cdots \to Q_3 \xrightarrow{q_3} Q_2 \xrightarrow{q_2} Q_1 \to 0\]
  where $Q_i=G/G_i$.  In particular, $Q_1$ is the abelianization of $G$.  By
  the work of Mal'cev \cite{Mal}, the $Q_i$ are likewise simply connected, and
  moreover, their Lie algebras are the lower central series quotients of the
  Lie algebra $\mathfrak g$ of $G$:
  \[\mathfrak q_i = \mathfrak g/\mathfrak g_i, \quad
    \text{where}\quad \mathfrak g_0 =\mathfrak g, \quad \mathfrak g_1=[\mathfrak g,\mathfrak g], \quad \mathfrak g_2=[\mathfrak g_1,\mathfrak g], \quad \ldots\]

  Now we have a sequence of maps
  \[H^*(Q_1) \xrightarrow{q_2^*} H^*(Q_2) \to H^*(G) \xrightarrow{i} \Lambda^*\RR^n\]
  whose composition is injective on $H^1(Q_1)$, and therefore on the subalgebra
  generated by $H^1(Q_1)$.  Since $Q_1$ is abelian, the differential in
  $\Omega^*(Q_1)^{Q_1}$ is zero, $H^1(Q_1)$ is all of the invariant $1$-forms,
  and so this subalgebra is all of $H^*(Q_1)$.  We will show that if $G$ is not
  abelian, that is if $Q_2 \neq Q_1$, then the induced homomorphism
  $q_2^*:H^2(Q_1) \to H^2(Q_2)$ cannot be injective, deriving a contradiction.

  Let $(q_2)_*$ be the induced quotient map $\mathfrak q_2 \to \mathfrak q_1$.
  By definition,
  \[\ker(q_2)_*=[\mathfrak q_2,\mathfrak q_2]\]
  and
  \[[\ker(q_2)_*,\mathfrak q_2]=[\mathfrak g_2,\mathfrak g]/\mathfrak g_3=0.\]
  From the duality between the Lie bracket and the differential, we get that
  \[\ker d|_{\Omega^1(Q_2)^{Q_2}}=q_2^*\Omega^1(Q_1)^{Q_1}\]
  and that for $\omega \in \Omega^1(Q_2)^{Q_2}$,
  \[d\omega(X,Y)=0 \qquad \text{for all }X \in \mathfrak q_2, Y \in \ker(q_2)_*.\]
  In other words, $d\Omega^1(Q_2)^{Q_2}$ is a $(\dim Q_2-\dim Q_1)$-dimensional
  subspace of $q_2^*\Omega^2(Q_1)^{Q_1}$.  On the other hand, since $Q_1$ is
  abelian, $H^*(Q_1)=\Omega^2(Q_1)^{Q_1}$, so this subspace is the kernel of
  $q_2^*$ in $H^2(Q_1)$.
\end{proof}

\begin{proof}[Proof of Theorem \ref{cor:pi1}]
  By Theorem \ref{thm:pi1}, $\pi_1(M)$ is virtually nilpotent.  Since covers of
  (quasiregularly, Gromov\=/)\allowbreak elliptic spaces are elliptic, we can
  assume by passing to a finite cover that $\pi_1(M)$ is nilpotent and
  torsion-free \cite[Lemma 2.42]{CSDA}.

  Let $\Gamma=\pi_1(M)$, and let $G$ be its \emph{Mal'cev completion}: a simply
  connected (and therefore contractible) nilpotent Lie group which contains it
  as a lattice \cite{Mal}.  Then the nilmanifold $N=G/\Gamma$ is an
  Eilenberg--MacLane space of type $K(\Gamma,1)$, so there is a map $g:M \to N$
  which induces a $\pi_1$-isomorphism.  The pullback and Theorem \ref{thm:coh}
  induce maps
  \begin{equation} \label{eqn:maps}
    H^*(N;\RR) \xrightarrow{g^*} H^*(M;\RR) \xhookrightarrow{\iota} \Lambda^*\RR^n
  \end{equation}
  whose composition is injective on $H^1(N;\RR)$, and therefore on the
  subalgebra generated by $H^1(N;\RR)$.  Therefore the subalgebra of
  $H^*(N;\RR)$ generated by $H^1(N;\RR)$ is isomorphic to $\Lambda^*\RR^k$ where
  $k=\rk H^1(N;\RR)$.

  Nomizu's theorem \cite{Nomizu} implies that $H^*(N;\RR)$ is isomorphic as a
  ring to $H^*(G)$.  Therefore, by Lemma \ref{lem:Lie}, $G$ is abelian, and so
  is its lattice $\Gamma$.
\end{proof}

So every elliptic and quasiregularly elliptic manifold has a finite cover whose fundamental group is $\ZZ^r$ for some $r$.

\subsection{Other topological restrictions}
We can use our results to derive further restrictions on the topology of elliptic and quasiregularly elliptic spaces:
\begin{cor} \label{cor:whole_torus}
  Suppose that $M$ is an elliptic or quasiregularly elliptic closed $n$-manifold and $\pi_1(M)$ has rank $k$.  Then there is finite cover $\tilde M$ of $M$ which admits an induced map $\tilde M \to T^k$ which is surjective on real homology.
\end{cor}
\begin{proof}
  First we construct the induced map.  Since $\pi_1(M)$ is virtually abelian, we can choose $\tilde M$ to be a finite cover such that $\pi_1(\tilde M) \cong \ZZ^k$.  Then there is a map $f:\tilde M \to T^k$ which induces an isomorphism on $\pi_1$.

  Now consider the pullback of $H^*(T^k;\RR)$ along $f$.  Since $H^*(\tilde M;\RR)$ embeds in $\Lambda^*\RR^n$, the image of this pullback is the subalgebra generated by a $k$-dimensional vector subspace of $\Lambda^1\RR^n$.  This subalgebra is isomorphic to $\Lambda^*\RR^k$, meaning that $f^*$ is injective.  Dually, $f$ is surjective on homology.
\end{proof}
\begin{ex}
  Construct an $n$-manifold $M$, $n \geq 4$, by embedding $(T^3)^{(2)}$ (the $2$-skeleton of $T^3$) in $\RR^{n+1}$ and letting $M$ be the boundary of a thickening of this embedding.  Then $M$ is not elliptic or quasiregularly elliptic by Corollary \ref{cor:whole_torus}.
\end{ex}
\begin{cor} \label{cor:n-1}
    Suppose that $M$ is an elliptic or quasiregularly elliptic closed $n$-manifold.  Then $\pi_1(M)$ cannot have rank exactly $n-1$.
\end{cor}
\begin{proof}
  Suppose $\pi_1(M)$ has rank $n-1$.  By passing to a finite cover, we can assume that $\pi_1(M) \cong \ZZ^{n-1}$.  Consider an injective map $i:H^*(M;\RR) \to \Lambda^*\RR^n$.  The image of $i$ contains an $(n-1)$-dimensional subspace of $\Lambda^1\RR^n$, without loss of generality $\langle dx_1,\ldots,dx_{n-1} \rangle$.  The product $dx_1 \wedge \cdots \wedge dx_{n-1}$ must pair nontrivially with a $1$-form in the image of $i$, which then has to contain a nonzero $dx_n$ term.  This is a contradiction.
\end{proof}
One might guess from the previous two corollaries that in fact every $n$-dimensional Poincar\'e duality algebra $A$ which embeds in $\Lambda^*\RR^n$ splits as a product $A=\Lambda^*\RR^k \otimes B$, where $B$ has no elements of degree $1$.  However, we have only been able to prove something weaker:
\begin{prop} \label{prop:quotient}
  Suppose that $A \subset \Lambda^* \RR^n$ is a graded subalgebra which
  satisfies Poincar\'e duality in degree $n$, and let $x \in A^1$.  Then
  $A/(x)$ also satisfies Poincar\'e duality, and moreover $A/(x)$ and $(x)$
  have the same dimension.
\end{prop}
By induction, one gets that $A/\Lambda^*A^1$ also satisfies Poincar\'e duality.
\begin{proof}
  Choose a basis $x_1,\ldots,x_n \in \RR^n$ such that $x=x_1$, and write
  $d\vol=x_1 \wedge \cdots \wedge x_n$.
    
  Let $\bar a \in A/(x_1)$ be a nonzero element of degree $k$.  We would like
  to show that there is a $\bar b \in A/(x_1)$ such that
  \[\bar a \bar b=\overline{x_2 \wedge \cdots \wedge x_n}.\]
  Choose a preimage $a \in A$ of $\bar a$.  Suppose first that $ax_1 \neq 0$.
  Then there is a $b \in A$ such that $ax_1b=d\vol$, and we can choose $\bar b$
  to be the equivalence class of $b$.
    
  Otherwise, $a=x_1a'$ for some $a' \in \Lambda^* \RR^n$ which is not
  necessarily in $A$.  We show that this cannot happen.  Consider the vector
  space $V$ of elements of $A^k$ that have this form.  We can choose a vector
  space $W \subset A^{n-k}$ such that the wedge product defines a nondegenerate
  bilinear form on $V \otimes W$.  In particular, multiplication by $x_1$
  defines a bijection $W \to x_1W$.  Now we can similarly build a vector space
  $V' \subset A^{k-1}$ which is Poincar\'e dual to $x_1W$.  The map $V' \to V$
  defined by multiplication by $x_1$ is injective, and hence surjective since
  $V'$ and $V$ have the same dimension.  Therefore $a$ is contained in the
  ideal $(x_1)$ and $\bar a=0$, contradicting our assumption.

  Finally, since $x_1^2=0$, multiplication by $x_1$ gives a well-defined
  surjective map from $A/(x_1)$ to $(x_1)$.  The above argument also shows that
  this map is also injective, so these two vector spaces have the same dimension.
\end{proof}
\begin{cor} \label{cor:chi=0}
  Every elliptic or quasiregularly elliptic manifold $M$ with infinite
  fundamental group has Euler characteristic zero.
\end{cor}
This is also proved by a different method in \cite{PG}.
\begin{proof}
  Choose an element $x \in H^1(M)$.  Then this follows immediately from the
  bijection between $A/(x)$ and $(x)$ in Proposition \ref{prop:quotient}.
\end{proof}
\begin{cor} \label{cor:4D}
  Every closed elliptic or quasiregularly elliptic $4$-manifold with infinite
  fundamental group is finitely covered by a manifold homeomorphic to $T^4$,
  $T^2 \times S^2$, or $S^1 \times S^3$.
\end{cor}
\begin{proof}
    Let $M$ be such a manifold.  By passing to a finite cover and applying Corollary \ref{cor:n-1} we may assume that $M$ has fundamental group $\ZZ^4$, $\ZZ^2$ or $\ZZ$.  By Corollary \ref{cor:chi=0}, $M$ has Euler characteristic zero.  If $\pi_1(M) \cong \ZZ^4$, then by \cite[Corollary 3.5.1]{Hill}, $M$ is aspherical and therefore homotopy equivalent to $T^4$.  It is then homeomorphic to $T^4$ by \cite[Theorem 2.16]{FJ}.  Otherwise, by \cite[Corollary 6.11.1]{Hill}, $M$ is determined by the total Stiefel--Whitney class of the bundle $S^2 \to M \to T^2$ or $S^3 \to M \to S^1$, respectively.  Since this class is of finite order, $M$ is finitely covered by either $T^2 \times S^2$ or $S^1 \times S^3$.
\end{proof}
This completes the classification of elliptic and quasiregularly elliptic $4$-manifolds up to finite covers.  If $\pi_1(M)$ is finite, then $M$ is finitely covered by a simply connected manifold which must come from a finite list of 19 manifolds given in \cite[Appendix A]{HP}.

\section{Invariance of ellipticity}

Here we restate and prove Theorem \ref{thm:Q-inv} and Theorem \ref{thm:TxS}.
\begin{repthm}{thm:TxS}
    Let $X=(T^{n-k} \times S^k) \mathbin{\#} \Sigma$, where $2 \leq k \leq n-1$ and $\Sigma$ is a simply connected rational homology sphere.  Then $X$ is elliptic.
\end{repthm}
\begin{proof}
  We will show that there is a map $f:T^{n-k} \times S^k \to X$ of positive
  degree.  Since $T^{n-k} \times S^k$ is elliptic, this completes the proof.
    
  First, note that $\Sigma \setminus \{\text{pt}\}$ is a simply connected space
  with finite homology groups.  By the rational Hurewicz theorem, it also has
  finite homotopy groups, in particular there is an $r$ such that the map
  $S^{n-1} \to \Sigma \setminus \{\text{pt}\}$ winding $r$ times around the
  puncture is nullhomotopic.

  Now consider a map
  \[\id \times f_r:T^{n-k} \times S^k \to T^{n-k} \times S^k\]
  where $f_r$ is a map of degree $r$.  Choose a regular value $y$ of $f_r$,
  whose preimage consists of points $y_1,\ldots,y_r$, and fix a point
  $x \in T^{n-k}$.  Choose a topological ball $B \subset S^k$ which contains all
  the $y_i$ and an $\epsi>0$ such that $B_\epsi(x) \subset T^{n-k}$ is also a
  ball.

  Without loss of generality, $X$ is built by cutting out an open ball $N$ around
  $(x,y)$ whose preimage under $\id \times f_r$ is contained in the interior of
  $B_\epsi(x) \times B$, then gluing on a punctured $\Sigma$.  Then we can
  define a map $f:T^{n-k} \times S^k \to X$ of degree $r$ as follows:
  \begin{itemize}
  \item For $(x,y) \notin B_\epsi(x) \times B$, set $f(x,y)=(x,f_r(y))$.
  \item Identify the ball $B_\epsi(x) \times B$ with the cone on its
    boundary, i.e.
    \[B_\epsi(x) \times B \cong \partial(B_\epsi(x) \times B) \times [0,1]/(\partial(B_\epsi(x) \times B) \times \{1\}).\]
    Define the map on this cone to be the concatenation of two homotopies:
    \[\id \times f_r|_{\partial(B_\epsi(x) \times B)} \xrelbar[\text{in }T^{n-k} \times S^k \setminus N]{\sim} \text{degree $r$ map to }\partial N \xrelbar[\text{in }\Sigma \setminus \text{ball}]{\sim} \text{constant}. \qedhere\]
  \end{itemize}
\end{proof}

\begin{repthm}{thm:Q-inv}
  If $M$ is a simply connected (or more generally nilpotent) closed manifold, then whether it is elliptic depends only on its rational homotopy type.
\end{repthm}
\begin{proof}
  Suppose that the $n$-manifold $M$ is elliptic, and let $f:\RR^n \to M$ be a wrapping map.  Suppose that $N$ has the same rational homotopy type as $M$.  We use Sullivan's model of rational homotopy theory (see \cite{SulLong,GrMo,FOT}).  The manifolds $M$ and $N$ have the same Sullivan minimal algebra $\mathcal A$ and minimal models $m_M:\mathcal A \to \Omega^*M$ and $m_N:\mathcal A \to \Omega^*N$.  By the shadowing principle (\cite[Theorem 4.1]{PCDF} for simply connected targets, \cite[Theorem 3.15]{Hansen} for nilpotent targets) the DGA homomorphism $f^*m_M:\mathcal A \to \Omega^*\RR^n$ extends to a DGA homomorphism $\Phi:\mathcal A \to \Omega^*(\RR^n \times [0,1])$ such that
  \begin{enumerate}[(i)]
  \item The forms in the image of $\Phi$ have bounded operator norm, and
  \item $\Phi|_{t=1}=g^*m_N$ for some Lipschitz map $g:\mathbb R^n \to N$.
  \end{enumerate}
  Now let $a \in \mathcal A$ be an element representing the fundamental
  cohomology class of $M$, and hence a multiple of that of $N$.  Then\
  $m_Ma-d\vol_M=d\alpha$ for some $\alpha \in \Omega^{n-1}(M)$, and $f$ is a
  wrapping map if and only if, by Stokes' theorem,
  \[\epsi R^n \leq \int_{B(R)} f^*d\vol_M=\int_{B(R)} f^*m_Ma-\int_{\partial B(R)} f^*\alpha=\int_{B(R)} f^*m_Ma+O(R^{n-1}).\]
  A similar condition holds for $g^*m_Na$.  Again by Stokes' theorem,
  \[\int_{B(R)} f^*m_Ma-\int_{B(R)} g^*m_Na=\int_{\partial B(R) \times [0,1]} \Phi(a)=O(R^{n-1}).\]
  It follows that $g$ is also a wrapping map.
\end{proof}

\section{Open surfaces} \label{S:open-d2}

In this section, we discuss the relationship between ellipticity and quasiregular ellipticity for surfaces.  Here quasiregular ellipticity is fully understood due to the uniformization theorem: every Riemannian surface is conformally equivalent to $\RR^2$, $\HH^2$, $S^2$, or one of their quotients.  Those that are quasiregularly elliptic are precisely those that are not conformally hyperbolic.

We will see that this is quite independent of whether the surface, if open, is Gromov-elliptic.  Thus we give examples of metrics on the plane that are conformally hyperbolic yet Gromov-elliptic, and that are parabolic (that is, conformally equivalent to the Euclidean plane) and yet not Gromov-elliptic.  In the next section, we will see that in fact every open manifold, including every open surface, has a Gromov-elliptic metric.

The main difficulty in analyzing such examples is determining whether they are Gromov-elliptic.  To determine whether a surface is quasiregularly elliptic, we use a criterion due to Ahlfors \cite{Ahl}: a metric on the plane is parabolic if and only if $\int_0^\infty \frac{dr}{L(r)}$ diverges, where $L(r)$ is the circumference of the geodesic circle of radius $r$ centered at some fixed point.  A special case of this was highlighted and given a simple proof by Milnor \cite{Mil}: a rotationally symmetric metric on the plane is parabolic if its curvature $K$ satisfies $K \geq -\frac{1}{r^2\log r}$ for large radii $r$, and hyperbolic if for some $\epsi>0$ it satisfies $K \leq -\frac{1+\epsi}{r^2\log r}$ for large radii $r$.

\subsection{A Gromov-elliptic surface which is conformally hyperbolic}
\begin{prop}
    Let $(\RR^2,g)$ be a Riemannian surface whose metric is symmetric about the origin and has curvature $K|_{\mathbf x}=-d(\mathbf x,0)^{-2}$ at every $\mathbf x \in \RR^2$.  Then $(\RR^2,g)$ is conformally hyperbolic, but Gromov-elliptic.
\end{prop}
One can define a metric on the plane with prescribed curvature $K(r)$ in polar coordinates by $dr^2+r^2\ph(r)^2d\theta^2$, where $\ph$ is the solution to the initial value problem $\ddot\ph=-K\ph$, $\ph(0)=1$.  When $K(r)$ is nonpositive, this is a well-defined Riemannian metric for all $r$.

Another way of getting a surface with these two properties would be to apply the construction in \S\ref{S:anyopen} to a punctured torus and take the universal cover of the resulting space.
\begin{proof}
    The surface is conformally hyperbolic by Milnor's result.

    Fix a unit-speed geodesic $\gamma$ starting at the origin, for example parametrizing the $y$-axis.  We first understand the exponential map $\ph_R$ at a point $x=\gamma(R)$, applied to a disk of radius $R/2$ around the origin.  The curvature on the image of this map is pinched between $-4R^{-2}$ and $-R^{-2}/4$.  Then by Toponogov's theorem, the exponential map from a Euclidean disk of radius $R/2$ to this disk is $4\sinh(1)$-Lipschitz.  On the other hand, the map clearly expands areas everywhere.

    We define a wrapping map $f:(\RR^2,g_{\text{Eucl}}) \to (\RR^2,g)$ as follows:
    \begin{itemize}
    \item For $n=1,2,\ldots$, define $f|B((0,2^{2n}),2^{2n-1})$ by 
      \[f(x,y)=\ph_{2^{2n-1}}(x,y-2^{2n}).\]
    \item For all points on the $y$-axis, send $(0,y) \mapsto \gamma(y)$.
    \item For all other points, compose with a Lipschitz projection to the boundary of the region on which $f$ is defined thus far.
    \end{itemize}
    Clearly the resulting map is Lipschitz and has positive asymptotic degree.
\end{proof}

\subsection{A parabolic surface which is not Gromov-elliptic}
\begin{prop}
    Let $\Sigma$ be the graph in $\RR^3$, with the induced metric, of a smooth function $h:\RR^2 \to \RR$ defined by
    \[h(x,y)=\log(\sqrt{a^2+b^2})\ph\bigl(d\bigl((x,y),\bigl(a+{\textstyle \frac12},b+{\textstyle \frac12}\bigr)\bigr)\bigr) \qquad a,b \in \ZZ,\:(x,y) \in [a,a+1] \times [b,b+1],\]
    where $\ph:[0,1] \to [0,1]$ is a smooth, non-increasing bump function supported on $[0,\frac{1}{2})$ and equal to $1$ on $[0,\frac{1}{4}]$.  Then $\Sigma$ is parabolic, but not Gromov-elliptic.
\end{prop}
Informally, $\Sigma$ is a plane with a lattice of spikes whose height is logarithmic in the distance from the origin.  The reason this space is not Gromov-elliptic is roughly that any map of positive asymptotic degree would be forced to have local winding number $\sim \log R$ around many different spikes; however, these local winding numbers add up to a global ``turning number'', forcing the restriction of such a map to a circle of circumference $R$ to make $\gtrsim R^2/\log R$ full turns each of which encompasses at least one spike.  This is an obvious impossibility.
\begin{proof}
  The parabolicity of $\Sigma$ can be seen using Ahlfors' criterion by
  considering its volume growth.  Clearly the $R$-ball in $\Sigma$ is contained
  in the image of the Euclidean $R$-ball.  The area of the image of the
  square $[a,a+1] \times [b,b+1]$ is bounded by the sum of its projections in
  the three coordinate directions.  These projections are contained in five
  sides of a rectangular prism of height $\log(\sqrt{a^2+b^2})$, giving a bound
  of $4\log(\sqrt{a^2+b^2})+1$.  Therefore
  \[\vol B_R(\Sigma) \leq \pi R^2+8\pi\int_0^R x\log(x+2)dx \leq 4\pi R^2\log(R+2).\]
  It follows that, for at least half of the $t \in [0,R]$, the circumference of
  the geodesic circle at $t$ is bounded by $4\pi R\log(R+2)$.  Since
  \[\int (R\log R)^{-1}dR=\log\log R,\]
  the integral of the circumferences diverges.

  To see that $\Sigma$ cannot be Gromov-elliptic, consider a $1$-Lipschitz map
  $f_R:B(R) \to \Sigma$ such that $|f_R(0)|<R$.  We will show that any such
  function has degree $O(R^2/\log R)$.  Now if $f_\infty:\RR^2 \to \Sigma$ is a
  $1$-Lipschitz map, for large enough $R$ the restriction $f_R=f_\infty|_{B(R)}$
  satisfies the assumption $|f_R(0)|<R$.  Thus $f_\infty$ cannot have positive
  asymptotic degree.

  Identify points in $\Sigma$ with their projection to the plane.  We can
  assume that $f_R$ is smooth and transverse to the curves
  \[C_{a,b}=\Bigl\{(x,y \in [a,a+1] \times [b,b+1] : \ph\bigl(d\bigl((x,y),\bigl(a+{\textstyle \frac12},b+{\textstyle \frac12}\bigr)\bigr)\bigr)=\frac{1}{2}\Bigr\}\]
  for every $a,b \in \ZZ$.

  Let $G$ be the CW complex whose cells are vertices, edges, and squares of the integer grid in $\RR^2$.
  \begin{lem}
    The function $f_R$ extends to a $\pi\sqrt 2$-Lipschitz function $f:B(R+\log R+2) \to \Sigma$ such that $f|_{\partial B(R+\log R+2)}$ maps to the set
    \[G^{(1)}=\{x,y \in \RR^2 : x \in \ZZ\text{ or }y \in \ZZ\}.\]
  \end{lem}
  \begin{proof}
    It suffices to build a $\pi\sqrt 2$-Lipschitz homotopy $h:\partial B(R) \times [0,\log R+2] \to \Sigma$ from $f_R|_{\partial B(R)}$ to a map with the desired property.  The domain of the homotopy is identified with the annulus in the obvious way.

    We build this homotopy in three steps.  First, we homotope $f_R$ out of the ``flat top'' $U_{a,b}=B_{1/4}(a+\frac{1}{2},b+\frac{1}{2})$ of each spike.  Since $U_{a,b}$ is an open set, its preimage under $f_R|_{\partial B(R)}$ is a disjoint union of open arcs.  On each such arc $A=(\alpha,\beta)$, we perform a linear homotopy from $f_R|_A$ to the constant-speed shortest path between the endpoints $f(\alpha)$ and $f(\beta)$ along the boundary of $U_{a,b}$.  This homotopy is $\frac\pi2$-Lipschitz and takes time $1/2$; call its end result $g:\partial B(R) \to \Sigma$.

    Next we push the curve down the side of the spike, as with a squeegee: a point starts moving once the homotopy reaches down to its height.  This homotopy takes time $<\log(2R)+\frac{1}{4}<\log R+1$ and is $\pi$-Lipschitz, since distances increase by a factor of at most $2$ from $g$.

    Finally, we push points from the circle of radius $1/2$ around $(a+\frac{1}{2},b+\frac{1}{2})$ linearly out to the square; at this point we accrue the factor of $\sqrt2$, and the homotopy takes time $\sqrt 2-1$.
  \end{proof}
  In particular, every $f^{-1}(C_{a,b})$ is a closed curve and $f$ has the same local degree over every point in a grid square.
    
  The domain can then be rescaled to make the function once again $1$-Lipschitz.  We will show that the resulting function satisfies
  \[\int_{B(\pi\sqrt2(R+\log R+2))} f^*d\vol_\Sigma=O(R^2/\log R),\]
  and therefore the original function $f_R$ likewise has degree
  $O(R^2/\log R)$.
    
  Now write $B=B(\pi\sqrt 2(R+\log R+2))$.
  \begin{lem} \label{lem:loc-degs}
    The sum of the local degrees of $f$ over points in
    \[\Bigl(\ZZ^2+\Bigl(\frac{1}{2},\frac{1}{2}\Bigr)\Bigr) \cap [-\sqrt{R},\sqrt{R}]^2\]
    is $O(R^{3/2})$.
  \end{lem}
  \begin{proof}
    Consider the homology class induced by $f(\partial B)$ in $G^{(1)}$.  This induces a unique cellular cycle $z \in C_1(G)$; the mass of $z$ (that is, the sum of the coefficients) is $O(\operatorname{length}(f(\partial B)))=O(R)$.  Let $c=\sum_{\text{cells }\alpha}c_\alpha\alpha$ be the (likewise unique) $2$-chain filling $z$; we must estimate the portion of the mass of $c$ contained in $[-\sqrt{R},\sqrt{R}]^2$.

    Let $c(i)=\{\text{2-cells }\alpha : c_\alpha \geq i\}$.  Then $\sum_{i=1}^\infty \#c(i) \geq \mass(c)$ and
    \[\mass(z) \geq \sum_{i=1}^\infty \#\partial c(i) \geq \sum_{i=1}^\infty 4\sqrt{\#c(i)},\]
    where the last inequality holds by the isoperimetric inequality for $(\RR^2,\ell^\infty)$.

    Now let $V$ be the portion of the mass of $c$ contained in $[-\sqrt{R},\sqrt{R}]^2$.  We have
    \[V \leq \sum_{i=1}^\infty \min\{\# c(i),4R\} \leq 2\sqrt{R}\sum_{i=1}^\infty \sqrt{\# c(i)} \leq \frac{1}{2}\sqrt{R}\mass(z)=O(R^{3/2}). \qedhere\]
  \end{proof}

  Define the \emph{turning number} $t(\gamma)$ of a non-backtracking piecewise
  linear path $\gamma$ in the plane to be $\frac{1}{2\pi}$ times the sum of its
  turning angles in $(-\pi,\pi)$; intuitively, this is the signed number of
  360\textdegree\ turns made by a person traveling along the path.  We can
  extend the definition to paths with backtracks by defining $t(\gamma)$ to be
  the turning number of the path with all backtracks pruned away.  Then the
  following is clear:
  \begin{lem} \label{lem:turning-additive}
    If $\alpha$ and $\beta$ are two piecewise linear loops, then
    \[\lvert t(\gamma)-t(\alpha)-t(\beta) \rvert \leq 1\]
    where $\gamma$ is any \emph{concatenation} of $\alpha$ and $\beta$, that
    is, any loop that traverses $\alpha$, then travels along a piecewise linear
    path $\delta$ from $\alpha$ to $\beta$, then traverses $\beta$, and then
    travels back along $\delta$.  The inequality does not depend on the path
    $\delta$.
  \end{lem}
  In particular, for a non-backtracking path in $G^{(1)}$, the turning number is
  \[\frac{\#(\text{left turns})-\#(\text{right turns})}{4}.\]
  Every free loop $\gamma:S^1 \to G^{(1)}$ is homotopic to a unique
  non-backtracking (``reduced'') loop, so $t(\gamma)$ is homotopy-invariant.
  
  Now we slightly modify the concept of turning number to be well-defined for
  free homotopy classes of loops in $G^{(1)} \cup [-\sqrt{R},\sqrt{R}]^2$,
  which we denote by $t_{\operatorname{sq}}$.  Given such a loop $\gamma$, we
  produce a canonical representative of the free homotopy class by replacing
  each section that traverses $[-\sqrt{R},\sqrt{R}]^2$ by a pair of straight
  segments meeting at the origin, then pruning away backtracks.  Note that this
  choice of canonical representative commutes with concatenation as defined in
  Lemma \ref{lem:turning-additive}.  Then $t_{\operatorname{sq}}(\gamma)$ is
  defined to be the turning number of this representative and Lemma
  \ref{lem:turning-additive} holds on the level of free homotopy classes.
  Moreover:
  \begin{lem} \label{lem:turning}
    For any loop $\gamma:S^1 \to G^{(1)} \cup [-\sqrt{R},\sqrt{R}]^2$, $t_{\operatorname{sq}}(\gamma) \leq \operatorname{length}(\gamma)/4$.
  \end{lem}
  \begin{proof}
    Every segment of length $1$ in $G^{(1)}$ contributes at most $1/4$ to the
    turning number.  Every path between distinct vertices of
    $[-\sqrt{R},\sqrt{R}]^2$, which has length at least $1$, likewise
    contributes at most $1/4$.
  \end{proof}
  \begin{lem} \label{lem:many-turns}
    The sum of the local degrees of $f$ over points in $\ZZ^2+(\frac{1}{2},\frac{1}{2})$ is $O(R^2/\log^2 R)$.
  \end{lem}
  \begin{proof}
    Write
    \[A=\RR^2 \setminus \Bigl(\ZZ^2+\Bigl(\frac{1}{2},\frac{1}{2}\Bigr)\Bigr) \cup [-\sqrt{R},\sqrt{R}]^2.\]
    Since $A$ deformation retracts to $G^{(1)} \cup [-\sqrt{R},\sqrt{R}]^2$, we
    can consider the $t_{\operatorname{sq}}$ numbers of loops in $A$.

    Let $\gamma_1,\ldots,\gamma_s$ be counterclockwise parametrizations of the
    components of $f^{-1}C_{a,b}$ for
    \[\bigl(a+{\textstyle\frac{1}{2}},b+{\textstyle\frac{1}{2}}\bigr) \notin [-\sqrt{R},\sqrt{R}]^2\]
    which enclose at least one preimage of the point
    $(a+\frac{1}{2},b+\frac{1}{2})$ and are not nested within another
    $\gamma_j$.  The turning numbers of the $f \circ \gamma_i$ are just their
    winding numbers around $(a+\frac{1}{2},b+\frac{1}{2})$.
    
    By definition, the $\gamma_i$ bound disjoint regions, and the region
    outside all of them maps under $f$ to $A$.  Fix a homotopy between
    $\partial B$ and a concatenation of the $\gamma_i$ which stays within this
    region.  Composing this homotopy with $f$ gives a homotopy in $A$ between
    $f|_{\partial B}$ and a concatenation of the $f \circ \gamma_i$.  Therefore
    \[t_{\operatorname{sq}}(f|_{\partial B}) \geq \sum_i t(\gamma_i)-\#\{\gamma_i\}\]
    by Lemma \ref{lem:turning-additive}.

    Each $\gamma_i$ winds around a region which contains a preimage
    $u \in f^{-1}(a+\frac{1}{2},b+\frac{1}{2})$.  Since $f$ is $1$-Lipschitz,
    the region contains a ball of radius $\frac{1}{2}\log R$ around $u$, and
    these regions are disjoint.  Therefore, the number of distinct $\gamma_i$
    is $O(R^2/\log^2 R)$.

    Since $f$ is $1$-Lipschitz, $f|_{\partial B}$ has length $O(R)$, so by Lemma
    \ref{lem:turning}, $t_{\operatorname{sq}}(f|_{\partial B})=O(R)$.  It follows
    that $\sum_i t(\gamma_i)=O(R^2/\log^2 R)$.  This gives the sum of local
    degrees of $f$ over points not in $[-\sqrt{R},\sqrt{R}]^2$.  Combining this
    with Lemma \ref{lem:loc-degs}, which covers the points inside this square,
    we get the result.
    \end{proof}

    Grid squares in the image of $f$ have area bounded by
    $4\log((\pi\sqrt 2+1)R)+1$, as explained at the beginning of the proof.
    Moreover, since we assumed that $f|_{\partial B}$ maps to $G^{(1)}$, $f$ has
    the same local degree over all points in a grid square.  By Lemma
    \ref{lem:many-turns},
    \[\int_B f^*d\vol_\Sigma \leq (4\log((\pi\sqrt 2+1)R)+1)O(R^2/\log^2 R)=O(R^2/\log R),\]
    as desired.
\end{proof}

\subsection{Some remarks on surfaces with nontrivial topology}

One application of Theorem \ref{thm:pi1-open}, proved in the next section, is that a Lipschitz map $\RR^2 \to S^2$ of positive asymptotic degree can miss at most two points, since $S^2$ minus three points has a free fundamental group, which has exponential growth.

Now we show:
\begin{prop} \label{lem:2d}
  There is a Lipschitz map $f:\RR^2 \to S^2$ which has positive asymptotic degree and misses two points.
\end{prop}
This is similar to the case of quasiregular maps---however, as we see in the next section, even in the setting of maps to the sphere missing points, the two cases diverge in higher dimensions.
\begin{proof}
Without loss of generality, the missing points are the north and south pole.  As part of the construction, we first give a map
\[f_0:[-1,1] \times [0,\infty) \to B_1(\RR^2) \setminus \{(0,0)\}\]
which maps the boundary of the domain to the boundary of the disk.  This map is defined explicitly on $[0,1] \times [0,\infty)$ by the polar coordinate formula
\[f_0(x,y)=(r(x,y), \theta(x,y)), \quad \text{where} \quad \biggl\{
  \begin{aligned}
    r(x,y) &= x+e^{-y}(1-x) \\
    \theta(x,y) &= y+\ln r(x,y).
  \end{aligned}\]
When $x<0$, we define $f_0(x,y)=(r(|x|,y),-\theta(|x|,y))$.  Notice that $f_0$ is Lipschitz since $\frac{\partial r}{\partial x}$, $\frac{\partial r}{\partial y}$, $r\frac{\partial \theta}{\partial x}$, and $r\frac{\partial \theta}{\partial x}$ are bounded.
\begin{figure}
    \centering
    \tikzset{
      graphs/.pic={
        \begin{scope}[rotate=90]
          \pgfplothandlerlineto
          \foreach \k in {1,2,3,4} {
            \pgfplotgnuplot[plot\k]{plot [y=\k:4.5] (exp(\k)-1)/(exp(y)-1)}
            \pgfusepath{stroke}
          }
        \end{scope}
        \draw (0,-0.5) -- (0,4.5);
      }
    }
    \begin{tikzpicture}
      \clip (-5.5,-0.5) rectangle (5.5,4.5);
      \foreach \x in {-6,-2,2} {
        \draw[line width=0.3cm, color=red!20] (\x+0.15,4.5) -- (\x+0.15,0.15) -- (\x+1.85,0.15) -- (\x+1.85,4.5);
        \draw[line width=0.3cm, color=blue!20] (\x+0.15,-0.5) -- (\x+0.15,-0.15) -- (\x+1.85,-0.15) -- (\x+1.85,-0.5);
        \draw[line width=0.3cm, color=blue!20] (\x+2.15,4.5) -- (\x+2.15,0.15) -- (\x+3.85,0.15) -- (\x+3.85,4.5);
        \draw[line width=0.3cm, color=red!20] (\x+2.15,-0.5) -- (\x+2.15,-0.15) -- (\x+3.85,-0.15) -- (\x+3.85,-0.5);
      }
      \foreach \x in {-5,-3,-1,1,3,5} {
        \pic at (\x,0) {graphs};
        \pic[xscale=-1] at (\x,0) {graphs};
      }
      \draw (-6,0) -- (6,0);
      \fill (-1,0) circle (0.05cm);
      \node[anchor=north west,inner sep=1pt] at (-1,0) {$(0,0)$};
    \end{tikzpicture}
    \caption{The domain of the map $f$.  The regions outlined in pink and blue map to the northern and southern hemisphere, respectively.  The black curves are preimages of the prime meridian, and each region bounded by the black curves is a homeomorphic preimage of the rest of the sphere.  The image is not to scale; it is vertically compressed by a factor of $2\pi$.}
    \label{fig:f}
\end{figure}
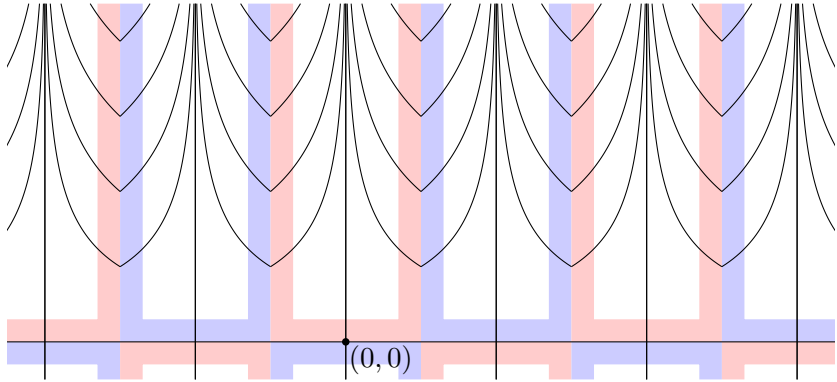

To define $f:\RR^2 \to S^2$, let $p_+,p_-:D^2 \to S^2$ be bilipschitz homeomorphisms of the disk onto the northern and southern hemisphere, respectively, which agree on $\partial D^2$.  Then we define $f$ everywhere by
\begin{align*}
    f(x,y) &= p_+(f_0(x,y)) & (x,y) &\in [-1,1] \times [0,\infty) \\
    f(x,y) &= p_-(f_0(2-x,y)) & (x,y) &\in [1,3] \times [0,\infty) \\
    f(x+4,y) &= f(x,y) \\
    f(x,-y) &= -f(x,y).
\end{align*}
For an illustration of this map, see Figure \ref{fig:f}.

Then:
\begin{itemize}
    \item $f$ is Lipschitz since $f_0$ is;
    \item $f$ has the same orientation at every regular point;
    \item The Jacobian of $f$ is bounded below by some $\epsi>0$ on the set
    \[\biggl(\bigcup_{z \in \ZZ} [2z+1/2,2z+3/2]\biggr) \times \bigl((-\infty,-1] \cup [1,\infty)\bigr).\]
\end{itemize}
The last two conditions imply that $f$ has positive asymptotic degree.
\end{proof}

\section{General open manifolds} \label{S:open}

Now we expand the discussion from surfaces to open manifolds of any dimension.  Here we give further constructions of manifolds which are elliptic, but not quasiregularly elliptic.  For example, we show that every open manifold has an elliptic metric, and at least a large class (all manifolds with a spherical end) have elliptic metrics of finite volume.  A great many of these manifolds admit no quasiregularly elliptic metric.  We also prove Theorem \ref{thm:pi1-open}, which gives a restriction on the fundamental group of an open elliptic manifold whose metric satisfies a certain rather strong assumption; this is analogous to a result for quasiregular ellipticity which requires no such assumption.  Nevertheless, as we show at the end of the section, even open subsets of compact manifolds (for which Theorem \ref{thm:pi1-open} holds) may be elliptic without being quasiregularly elliptic:
\begin{thm} \label{codim n-2}
  Let $n \geq 2$, and consider $S^{n-2}$ embedded in the round $S^n$ as the intersection with an $(n-1)$-plane through the origin in $\RR^{n+1}$.  The open manifold $S^n \setminus S^{n-2}$ is Gromov-elliptic, and therefore so is $S^n \setminus X$ for any $X \subseteq S^{n-2}$.
\end{thm}
\noindent In contrast, for any countably infinite set $X$, a generalization of the Rickman--Picard theorem \cite{Rick} by Holopainen and Rickman \cite[Theorem 3.1]{HoR2} shows that $S^n \setminus X$ does not admit a metric under which it is quasiregularly elliptic.

\subsection{An elliptic metric on any open manifold} \label{S:anyopen}
Let $M$ be any open $n$-manifold and choose any complete Riemannian metric $g$.  We will show that the metric $g$ can be modified to construct an elliptic metric, essentially by cutting it open and splicing in a copy of Euclidean space.

Let $\gamma:[0,\infty) \to M$ be a ray, that is, a proper embedding parametrized by arclength.  We modify $(M,g)$ by cutting out a small tubular neighborhood of $\gamma([0,\infty))$ and replacing it with an isometric copy of $\RR^n$ missing a small tubular neighborhood of the negative $x_1$-axis.  After a bit of smoothing, we get a new Riemannian manifold $(M',g')$ which is obviously diffeomorphic to $M$ and which includes an isometric copy of $[1,\infty) \times \RR^{n-1}$.  But then the map
\[f(x_1,\ldots,x_n)=\begin{cases}
    (x_1,\ldots,x_n) & x_1 \geq 1 \\
    (1, x_2, \ldots, x_n) & x_1 \leq 1
\end{cases}\]
is a wrapping map for $M'$.

\subsection{A finite-volume elliptic manifold with any fundamental group} \label{S:fvopen}
Let $(M,g)$ be any Riemannian $n$-manifold of finite volume.  We will show that by puncturing $M$ and modifying the metric in a neighborhood of the puncture, we can construct an elliptic manifold whose volume is still finite.  In particular, if $\dim M \geq 3$, this preserves the fundamental group.  This construction also produces examples of finite-volume elliptic surfaces whose fundamental group is a free group on any number of generators.

Specifically, we build a metric $\tilde g$ on $C=S^{n-1} \times [0,\infty)$ by embedding it in $\RR^{n+1}$ as a hypersurface of revolution obtained by revolving the graph of a function $\rho(t)$.  We define $\rho$ so that for every $p \in \NN$, it has a local minimum $\rho(p)=2^{-\frac{2p}{n-1}}$ at $p$, and a local maximum $\rho(p+1/2)=2^{-\frac{p}{n-1}}$; between these, the value of $\rho$ is interpolated smoothly so that the average value over the interval is proportional to the local maximum.  Thus each ``nodule'' $N_p$ from $p$ to $p+1$ has volume proportional to $2^{-p}$, and the total volume of $(C,\tilde g)$ is finite.  In between the nodules, at $p$, the metric has a spherical neck of surface area proportional to $2^{-2p}$.

We now describe a wrapping map $f:\RR^n \to (C,\tilde g)$.  It will be defined on a sequence of cubes
\[Q_p=[-2^{p-1},2^{p-1}]^{n-1} \times [2^p,2^{p+1}].\]
These cubes fill a proportion of roughly $2^{-n}$ of any ball around the origin, and thus it is sufficient for the restriction to them to have positive asymptotic degree.  On the portion of $\RR^n$ outside $\bigcup_p Q_p$, we define $f$ by first projecting to the boundary of $\bigcup_p Q_p$ (via a $\sqrt n$-Lipschitz map) and then applying the restriction to the relevant $Q_p$.

Inside $Q_p$, we choose a submanifold $S$ isometric to a $(n-1)$-dimensional rectangle which is folded up like an accordion, with $\partial S \subseteq \partial Q_p$, and so that the exponential map on its normal $1$-disk bundle is an embedding and fills up a proportion of $Q_p$ which depends only on $n$.  Restricting this exponential map $S \times [-\frac{1}{2},\frac{1}{2}]$ gives a codimension 0 submanifold $Q_p^0 \subseteq Q_p$ which is $2$-bilipschitz to $S \times [-\frac{1}{2},\frac{1}{2}]$ and separates the rest of $Q_p$ into two components $Q_p^+$ (above) and $Q_p^-$ (below).

Inside $S$, we choose a net of $\sim 2^{(n+1)p}$ balls $\{B_i\}$ of radius $2^{-\frac{p}{n-1}}$.  We then define a map to $C$ which sends $S \times [-\frac{1}{2},\frac{1}{2}]$ to $N_p$:
\[f(x,t)=\bigl(\rho\bigl(p+t+{\textstyle\frac{1}{2}}\bigr)\theta(x),p+t+{\textstyle\frac{1}{2}}\bigr),\]
where $\theta(x)=(1,0,\ldots,0)$ for $x$ outside the $B_i$ and $\theta$ maps each $B_i$ homeomorphically to $S^{n-1} \setminus (1,0,\ldots,0)$.  This map can now be extended to the rest of $Q_p$ so that $Q_p^-$ is mapped to the neck at $p$ and $Q_p^+$ is mapped to the neck at $p+1$, with degree $\sim 2^{(n+1)p}$ on the upper and lower boundary.  This extension can be made $O(1)$-Lipschitz because the cross-section of the neck has area $2^{-p}$ times the cross-sectional area of the center of a nodule; this allows all the homeomorphic preimages of the neck to be routed through the space outside $Q_p^0$.  See Figure \ref{fig:Q_p} for an illustration.
\begin{figure}
    \centering
    \begin{tikzpicture}
      \draw[very thick] (0,0) rectangle (7,7);
      \draw[thick] (0.25,6) arc (180:0:0.75);
      \draw[thick] (0.75,6) arc (180:0:0.25);
      \draw[thick] (2.25,6) arc (180:0:0.75);
      \draw[thick] (2.75,6) arc (180:0:0.25);
      \draw[thick] (4.25,6) arc (180:0:0.75);
      \draw[thick] (4.75,6) arc (180:0:0.25);
      \draw[thick] (1.25,1) arc (180:360:0.75);
      \draw[thick] (1.75,1) arc (180:360:0.25);
      \draw[thick] (3.25,1) arc (180:360:0.75);
      \draw[thick] (3.75,1) arc (180:360:0.25);
      \draw[thick] (5.25,1) arc (180:360:0.75);
      \draw[thick] (5.75,1) arc (180:360:0.25);
      \foreach \x in {0.25,0.75,...,6.75} {
        \draw[thick] (\x,1)--(\x,6);
      }
      \draw[thick] (.75,1) arc (0:-90:0.75);
      \draw[thick] (.25,1) arc (0:-90:0.25);
      \draw[thick] (6.25,6) arc (180:90:0.75);
      \draw[thick] (6.75,6) arc (180:90:0.25);
      \draw (2,7) .. controls (2,4) .. (2.25,4)--(2.75,4) .. controls (3,4) .. (3,0) (2.04,7) .. controls (2.04,4.25) .. (2.25,4.25)--(2.75,4.25) .. controls (3.04,4.25) .. (3.04,0);
      \draw (2.1,7) .. controls (2.1,4.75) .. (2.25,4.75)--(2.75,4.75) .. controls (3.1,4.75) .. (3.1,0) (2.14,7) .. controls (2.14,5) .. (2.25,5)--(2.75,5) .. controls (3.14,5) .. (3.14,0);
      \draw (1.9,7) .. controls (1.9,3.25) .. (2.25,3.25)--(2.75,3.25) .. controls (2.9,3.25) .. (2.9,0) (1.94,7) .. controls (1.94,3.5) .. (2.25,3.5)--(2.75,3.5) .. controls (2.94,3.5) .. (2.94,0);
      \draw (4,6.5) node {$Q_p^+$};
      \draw (4.5,3.5) node {$Q_p^0$};
      \draw (5,0.5) node {$Q_p^-$};
    \end{tikzpicture}
    \caption{A schematic illustration of $Q_p$, with the embedding of $S \times [-\frac{1}{2},\frac{1}{2}]$ and several preimages of the corresponding nodule $N_p$ shown.}
    \label{fig:Q_p}
\end{figure}
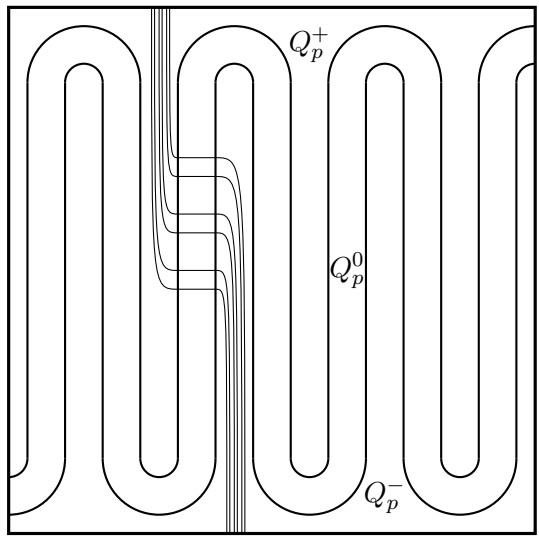

The last step is to match the map in the middle section of the bottom surface of $Q_p$ with that of the top surface of $Q_{p-1}$.  If we don't modify the map further, these two faces are covered approximately uniformly by preimages of the neck at $p$.  The top surface of $Q_{p-1}$ maps to the neck at $p$ with degree $2^{(n+1)(p-1)}$.  The middle section maps of the bottom surface of $Q_{p-1}$ with degree $2^{(n+1)p}/2^{n-1}=4 \cdot 2^{(n+1)(p-1)}$.  Thus in order to match the preimages one-to-one we need to route some of the preimages away to reduce the density in the middle section.  This can be done with Lipschitz constant $O(1)$ as long as we have $\sim 2^p$ vertical space, for example if $S$ does not intersect the bottom $10\%$ of $Q_p$.

Since the volume of the nodule $N_p$ is about $2^{-p}$, and $Q_p$ is mapped $\sim 2^{(n+1)p}$ times over $N_p$, the degree of $f|_{Q_p}$ is proportional to the volume of $Q_p$.  Therefore $f$ has positive asymptotic degree.

\subsection{Proof of Theorem \ref{thm:pi1-open}}
Now we prove a restriction on the fundamental group of elliptic manifolds satisfying certain hypotheses which exclude the behavior in the previous construction.  Effectively, these are manifolds over which any wrapping map equidistributes, that is, it must hit every point roughly equally often.  This condition holds for quasiregular maps to any manifold, so it is not surprising that it shows up here.

\begin{repthm}{thm:pi1-open}
Suppose that $M$ is a Riemannian $n$-manifold of finite volume, and suppose that the volume form of $M$ can be written as $\omega+d\alpha$ where $\omega$ is compactly supported and $\alpha$ is bounded.  If $M$ is elliptic, then any finitely generated subgroup of it has growth $O(R^n)$.
\end{repthm}
\noindent Note that $\pi_1(M)$ may be an infinitely generated group, for example $\QQ$.  Each of its finitely generated subgroups is virtually nilpotent by Gromov's polynomial growth theorem, and the nilpotency class is bounded by $n$ (and in fact around $\sqrt n$).  If one could bound the index of the nilpotent subgroups, then $\pi_1(M)$ itself would be virtually nilpotent.

The hypothesis holds for example for manifolds of pinched negative curvature as well as real analytic manifolds of nonpositive curvature, including locally symmetric spaces.  The proof in both cases is the same and follows from results described in \cite[\S10, 13]{BGS}.

We now relate the hypothesis of Theorem \ref{thm:pi1-open} to a more familiar and geometrically intuitive condition.  If $M$ is a Riemannian $n$-manifold of finite volume (not necessarily complete), the \emph{Cheeger isoperimetric constant} of $M$ is defined by
\[h(M) = \inf \left\{\frac{\vol_{n-1}(\partial A)}{\min\{\vol_n A,\vol_n A^c\}} : A \subset M \text{open}\right\}.\]
When $M$ is noncompact, an equivalent (and more commonly used) definition is
\[h(M)=\inf \left\{\frac{\vol(\partial A)}{\vol A} : A\text{ is a relatively compact open subset with }\vol A \leq \frac{1}{2}\vol M\right\}.\]
To see the equivalence, note that in the first definition one can make sure that the smaller subset is relatively compact (without changing the volumes too much) by intersecting it with a large ball around some base point; the coarea formula ensures that one can pick the radius of the ball so that $\vol(\partial A)$ is affected an arbitrarily small amount.

The Cheeger constant of a manifold can be zero, for example for those constructed in Section \ref{S:fvopen}.  In future work, we will show that $M$ has nonzero Cheeger constant if and only if it satisfies the hypothesis in Theorem \ref{thm:pi1-open}.  This equivalence has an analogue for manifolds of infinite volume stated by Gromov \cite[\S4.1]{GrHMGA} which can be proved using similar techniques.

\begin{proof}[{Proof of Theorem \ref{thm:pi1-open}.}]
  Let $f:\RR^n \to M$ be a wrapping map.  Write $d\vol_M=\omega+d\alpha$, where
  $\omega$ is compactly supported and $\alpha$ is bounded.  Then by Stokes'
  theorem,
  \[\int_{B(R)} f^*\omega=\int_{B(R)} f^*d\vol_M-\int_{\partial B(R)} f^*\alpha,\]
  and therefore
  \[\left\lvert\int_{B(R)} f^*\omega\right\rvert \geq \left\lvert\int_{B(R)} f^*d\vol_M \right\rvert-O(R^{n-1}).\]
  In particular, suppose $\omega$ is supported on a compact set $K$.  Then
  \[\limsup_{R \to \infty} \vol(B(R) \cap f^{-1}(K)) \geq \epsi R^n\]
  for some $\epsi>0$.  Choose a sequence $R_j \to \infty$ such that
  \[\vol(B(R_j) \cap f^{-1}(K)) \geq \epsi R_j^n.\]

  Now, we have freedom to choose $K$ as large as we want, so let's assume that
  \begin{enumerate}[(i)]
  \item $K$ is a connected, smooth codimension zero submanifold of $M$.
  \item $K$ contains loops generating $H$, where $H$ is some finitely generated
    subgroup of $\pi_1(M)$.
  \end{enumerate}
  Moreover, without loss of generality, $f$ is smooth and transverse to
  $\partial K$.  Let $i:K \to M$ be the inclusion map.  We will use a version
  of the Varopoulos isoperimetric inequality, as proved by Gromov in
  \cite[\S6.E]{GrMS}, to bound the growth of the finitely generated subgroup
  $G=i_*\pi_1(K) \subseteq \pi_1(M)$, which contains $H$.  This inequality
  relates the growth of a finitely generated group to the isoperimetry of
  subsets of its Cayley graph.  The idea of the argument is rather simple, but
  some technicalities seem to be unavoidable.  We choose to use arguments from
  geometric measure theory; a more elementary, but messier argument would use
  an approximation of $f|_{f^{-1}K}$ by a simplicial map.

  The map $f$ lifts to the universal cover of $M$.  Let $\tilde K$ be a
  connected component of the preimage of $K$ under the covering map; this is a
  cover of $K$ with deck group $\pi_1(K)/\ker i_* \cong G$, where $i:K \to M$
  is the inclusion map.  We lift $f$ to a (componentwise $1$-Lipschitz) map
  $\tilde f:f^{-1}K \to \tilde K$ by choosing a base point for every connected
  component.  Now consider the integral $n$-current
  \[E_j=\tilde f_*[f^{-1}K \cap B(R_j)] \in \mathbf I_n(\tilde K,\partial \tilde K).\]
  Then $\mass(E_j) \geq \epsi R_j^n$ by definition, and
  \[\mass(\partial E_j)=\mass \tilde f_*(f^{-1}K \cap \partial B(R_j)) \leq \mass(\partial B(R_j))=O(R_j^{n-1}).\]

  Choose a (finite) triangulation $\mathcal T$ of $K$; this lifts to a
  $G$-invariant triangulation $\widetilde{\mathcal T}$ of $\tilde K$.  By a
  version of the Federer--Fleming deformation theorem, see
  \cite[Thm.~10.3.3]{ECHLPT}, we can write $E_j=P_j+Q_j$, where $P_j$ is a
  simplicial $n$-chain in $\widetilde{\mathcal T}$ and
  \begin{align*}
    \mass(P_j) &\leq C(\mathcal T)\mass(E_j), \\
    \mass(Q_j), \mass(\partial P_j) &\leq C(\mathcal T)\mass(\partial E_j).
  \end{align*}
  In particular, there is some $n$-simplex $\Delta \in \mathcal T$ such that
  for infinitely many $j$, the total multiplicity of $P_j$ on lifts of $\Delta$
  is proportional to $R_j^n$.

  Let $S$ be a finite generating set for $G$.  Choose a basepoint
  $x \in \Delta$ and loops representing the elements of $S$ so that the
  resulting graph is transverse to $\mathcal T$.  Lifting this to $\tilde K$
  gives an embedding in $\tilde K$ of the Cayley graph $\Gamma$ of $G$ which is
  transverse to $\widetilde{\mathcal T}$.  Now we define a cochain
  $\sigma_j \in C^0(\Gamma)$ by
  \[\sigma_j(v)=\text{multiplicity of $P_j$ on }\Delta.\]
  Then $\lVert\sigma_j\rVert_1=\Theta(R_j^n)$.  Moreover, the value of
  $\delta\sigma_j \in C^1(\Gamma)$ on an edge of $\Gamma$ is the intersection
  number of the corresponding path in $\tilde K$ with $\partial P_j$, and
  therefore $\lVert \delta\sigma_j \rVert_1=O(R_j^{n-1})$.

  We now recall Gromov's argument to show that $G$ has growth $O(R^n)$.  Write
  $\sigma_j \cdot g$ for the right action of $g \in G$ on $\sigma_j$ (where the
  invariant action of $G$ on $\tilde K$ and $\Gamma$ is from the left).  For
  $s \in S$, this action moves each vertex along the adjacent edge
  corresponding to $s$, and we have
  \[\lVert \sigma_j \cdot s-\sigma_j \rVert_1=\sum_{\text{lifts $\tilde s$ of }s} \lvert \delta\sigma_j(\tilde s) \rvert \leq \lVert \delta\sigma_j \rVert_1.\]
  Therefore, more generally, writing $\lvert g \rvert$ for the word length of
  $g \in G$ in terms of $S$, we get that
  \[\lVert \sigma_j \cdot g-\sigma_j \rVert_1 \leq \lvert g \rvert \cdot \lVert \delta\sigma_j \rVert_1.\]
  Let $r$ be minimal such that the ball $B_r(G)$ in the word length metric
  satisfies $\#B_r(G) \geq 2\lVert\sigma_j\rVert_1$.  Then for every
  $a \in \supp(\sigma_j)$, for at least half of the $g \in B_r(G)$, $a \cdot g$
  is not in the support of $\sigma_j$.  Averaging the previous inequality over
  $B_r(G)$, we see that
  \[\frac{1}{2}\lVert \sigma_j \rVert_1 \leq r\lVert \delta\sigma_j \rVert_1,\]
  and therefore $r=\Omega(R_j)$ and
  \[\#B_r(G) \leq \# S \cdot \lVert \sigma_j \rVert_1=O(R_j^n)=O(r^n).\]
  Since this is true for a sequence of $R_j \to \infty$, $G$ has growth
  $O(R^n)$.

  Recall that $G$ contained $H$, an arbitrarily chosen finitely generated
  subgroup of $\pi_1(M)$.  Therefore $H$ also has growth $O(R^n)$.
\end{proof}

\subsection{Proof of Theorem \ref{codim n-2}}
Proposition \ref{lem:2d} shows that the theorem holds for $n=2$.  We will use the intermediate map $f_0$ defined in that lemma to construct maps $\RR^n \to S^n \setminus S^{n-2}$ of positive asymptotic degree for $n \geq 3$.
    
First define a Lipschitz map of positive asymptotic degree $g_{n-2}:\RR^{n-2} \to S^{n-2}$; for example, we can compose the universal covering map $\RR^{n-2} \to T^{n-2}$ with a map that collapses the $(n-3)$-skeleton of the torus.

Thinking of $S^k$ as the unit sphere in $\RR^{k+1}$, define a quotient map $q:D^2 \times S^{n-2} \to S^n$ by
\[q(\mathbf x,\mathbf z)=\bigl(p(\mathbf x),\sqrt{1-\lVert p(\mathbf x) \rVert^2}\mathbf z\bigr),\]
where $p:D^2 \to D^2$ is a homeomorphism defined so that $q|_{D^2 \times \{\mathbf z\}}$ is a bilipschitz embedding.  Finally let $\rho:\RR^{n-2} \to \RR^{n-2}$ be an orientation-reversing map (say, flipping one coordinate).

Now we define our desired map similarly to the way $f$ was defined above:
\begin{align*}
    f_n(x,y,\mathbf z) &= q(f_0(x,y),g_{n-2}(\mathbf z))
    & (x,y) &\in [-1,1] \times [0,\infty) \times \RR^{n-2} \\
    f_n(x,y,\mathbf z) &= q(f_0(2-x,y),g_{n-2}\rho(\mathbf z))
    & (x,y) &\in [1,3] \times [0,\infty) \times \RR^{n-2} \\
    f_n(x+4,y,\mathbf z) &= f_n(x,y,\mathbf z) \\
    f(x,-y,\mathbf z) &= f_n(x,y,\rho(\mathbf z)).
\end{align*}
This map is continuous since on the boundaries between subdomains, the map depends only on $x$ and $y$.  As in the two-dimensional case, a computation shows that $f_n$ is Lipschitz, has the same orientation at every regular point, and that its Jacobian is bounded away from zero on, e.g., the set
\[\biggl(\bigcup_{z \in \ZZ} \biggl[4z+\frac 14,4z+\frac 34\biggr]\biggr) \times [1,\infty) \times \{\mathbf z \in \RR^{n-2}: J_{g_{n-2}}(z)>\epsi\},\]
for some $\epsi>0$.  This implies that it is of positive asymptotic degree.

\subsection*{Acknowledgements}
The authors would like to thank Ben Lowe for comments drawing attention to the parallels between their previous work, which sparked this collaboration.  We would also like to thank Pekka Pankka and Susanna Heikkil\"a for useful conversations and comments, and Shmuel Weinberger for suggesting the proof of Theorem \ref{thm:TxS}.  Finally we would like to thank an anonymous referee for a very detailed reading that greatly improved the exposition.

\bibliographystyle{amsplain}
\bibliography{ellipticity}
\end{document}